\documentclass[12pt,draftcls,onecolumn]{IEEEtran}

\def\bsg{{\boldsymbol{g}}}
\def\bsh{{\boldsymbol{h}}}

\def\bsu{{\boldsymbol{u}}}
\def\bsv{{\boldsymbol{v}}}

\def\bsx{{\boldsymbol{x}}}

\def\bsH{{\boldsymbol{H}}}

\def\bsK{{\boldsymbol{K}}}
\def\bsL{{\boldsymbol{L}}}

\def\bsM{{\boldsymbol{M}}}

\def\bsQ{{\boldsymbol{Q}}}

\def\bsW{{\boldsymbol{W}}}

\def\calA{{\mathcal{A}}}

\def\calG{{\mathcal{G}}}

\def\calP{{\mathcal{P}}}

\usepackage{amssymb}
\usepackage{algorithm,algorithmic}

\usepackage{tikz}
\usepackage{tkz-berge}
\usepackage{lipsum}
\usetikzlibrary{calc}
\usetikzlibrary{shapes,fit}
\usepackage{verbatim}
\usepackage{setspace}
\usepackage{multirow}
\usetikzlibrary{shapes,arrows}
\usepackage[short]{optidef}

\usepackage{amsthm}

\usepackage{color}
\definecolor{gray}{RGB}{128,128,128}

\newcommand{\red}[1]{{\color{red} #1}}

\usepackage{cite}
\usepackage{graphicx}
\graphicspath{{figure/}}
\usepackage{epstopdf}
\usepackage{tikz}
\usetikzlibrary{arrows,shapes,backgrounds}

\usepackage{bm}

\usepackage{array}
\newcolumntype{M}[1]{>{\centering\arraybackslash}m{#1}}
\newcolumntype{N}{@{}m{0pt}@{}}

\usepackage{multirow}

\usepackage{stfloats}
\usepackage{caption}
\usepackage{subcaption}

\newtheorem{theorem}{Theorem}
\newtheorem{assumption}{Assumption}

\newtheorem{lemma}{Lemma}

\newtheorem{definition}{Definition}
\newtheorem{remark}{Remark}

\DeclareMathOperator{\nullrank}{null}

\DeclareMathOperator{\col}{col}

\DeclareMathOperator{\diag}{diag}
\DeclareMathOperator{\Deg}{Deg}

\DeclareMathOperator*{\argmin}{arg\, min}

\newenvironment{proof}[1][Proof]%
  {\smallskip\par\noindent\textbf{#1\,:\ }}%
  {\hspace*{\fill} \rule{6pt}{6pt}\smallskip}
\newenvironment{proof*}[1][Proof]%
  {\smallskip\par\noindent\textbf{#1\,:\ }}%

\newlength{\fwidth}

\allowdisplaybreaks
\begin{document}
\IEEEoverridecommandlockouts
\title{Linear Convergence of First- and Zeroth-Order Primal--Dual Algorithms for Distributed Nonconvex Optimization}
\author{Xinlei Yi, Shengjun Zhang, Tao Yang, Tianyou Chai, and Karl H. Johansson
\thanks{This work was supported by the
Knut and Alice Wallenberg Foundation, the  Swedish Foundation for Strategic Research, the Swedish Research Council, the National Natural Science Foundation of China under grants 61991403, 61991404, and 61991400, and the 2020 Science and Technology Major Project
of Liaoning Province under grant 2020JH1/10100008.  A preliminary version of this paper has
been accepted at the 59th IEEE Conference on Decision and Control, December 14-18, 2020, Jeju Island, Republic of Korea.}
\thanks{X. Yi and K. H. Johansson are with the Division of Decision and Control Systems, School of Electrical Engineering and Computer Science, KTH Royal Institute of Technology, and they are also affiliated with Digital Futures, 100 44, Stockholm, Sweden. {\tt\small \{xinleiy, kallej\}@kth.se}.}%
\thanks{S. Zhang is with the Department of Electrical Engineering, University of North Texas, Denton, TX 76203 USA. {\tt\small  ShengjunZhang@my.unt.edu}.}
\thanks{T. Yang and T. Chai are with the State Key Laboratory of Synthetical Automation for Process Industries, Northeastern University, 110819, Shenyang, China. {\tt\small \{yangtao,tychai\}@mail.neu.edu.cn}.}
}

\maketitle

\begin{abstract}                
This paper considers the distributed nonconvex optimization problem of minimizing a global cost function formed by a sum of local cost functions by using local information exchange.  We first propose a distributed first-order primal--dual algorithm. We show that it converges sublinearly to a stationary point if each local cost function is smooth and linearly to a global optimum under an additional condition that the global cost function satisfies the Polyak--{\L}ojasiewicz condition. This condition is weaker than strong convexity, which is a standard condition for proving linear convergence of distributed optimization algorithms, and the global minimizer is not necessarily unique. Motivated by the situations where the gradients are unavailable, we then propose a distributed  zeroth-order algorithm, derived from the proposed distributed first-order algorithm by using a deterministic gradient estimator, and show that it has the same convergence properties as the proposed first-order algorithm  under the same conditions. The theoretical results are illustrated by numerical simulations.

\end{abstract}



\section{Introduction}
Distributed convex optimization has a long history, which can be traced back at least to the 1980's \cite{tsitsiklis,tsitsiklis1986distributed,bertsekas1989parallel}.
It has gained renewed interests in recent years due to its wide applications in power systems, machine learning, and sensor networks, just to name a few \cite{nedich2015convergence,yang2019survey}.
Various distributed optimization algorithms have been developed. Basic convergence results in distributed convex optimization typically ensure that algorithms converge to optimal points sublinearly, e.g., \cite{johansson2008subgradient,Nedic09,zhu2011distributed,Nedic15, Yang-EDP-Delay}. Linear convergence rate can be established under more stringent strong convexity conditions. For example, in \cite{lu2012zero,kia2015distributed,shi2014linear,ling2015dlm,jakovetic2015linear,makhdoumi2017convergence,
berahas2018balancing,mokhtari2016dqm,nedic2017achieving,qu2017accelerated,
qu2018harnessing,jakovetic2019unification,mansoori2019flexible,xi2018add,
xu2018convergence} and \cite{Damiano-TAC2016,maros2019q,tian2018asy}, the authors assumed that each local cost function and the global cost function are strongly convex, respectively.

Unfortunately, in many practical applications, such as least squares, the cost functions are not strongly convex \cite{yang2018distributed}. This situation has motivated researchers to consider alternatives to strong convexity. There are some results in centralized optimization. For instance, in \cite{necoara2019linear}, the authors derived linear convergence of several centralized first-order methods for smooth and constrained optimization problems when cost functions are convex and satisfy  the quadratic functional growth condition; and in \cite{karimi2016linear}, the authors showed linear convergence of centralized proximal-gradient methods for  smooth optimization problems when cost functions satisfy the Polyak--{\L}ojasiewicz (P--{\L}) condition  which is weaker than the conditions assumed in \cite{necoara2019linear}. There also are some results in distributed optimization \cite{shi2015extra,zeng2017extrapush,xi2017dextra,Yi2018distributed,liang2019exponential,
yi2019exponential}. Specifically, in  \cite{shi2015extra}, the authors proposed the distributed exact first-order algorithm (EXTRA) to solve smooth convex optimization problems and proved linear convergence under the conditions that the global cost function is restricted strongly convex and the optimal set is a singleton,  which are stronger than the P--{\L} condition. The authors of \cite{zeng2017extrapush,xi2017dextra} later extended the results in \cite{shi2015extra} to directed graphs. In \cite{Yi2018distributed}, the authors proposed a continuous-time distributed heavy-ball algorithm with event-triggered communication to solve smooth convex optimization problems and proved exponential convergence under the same conditions as that assumed in \cite{shi2015extra}.  In \cite{liang2019exponential}, the authors established linear convergence of the distributed primal--dual gradient descent algorithm for solving smooth convex optimization problems under the condition that the primal--dual gradient map is metrically subregular, which is different from the P--{\L} condition but weaker than strong convexity. In \cite{yi2019exponential}, the authors proposed a distributed primal--dual gradient descent algorithm to solve smooth optimization problem problems and established linear convergence under the assumptions that the global cost function satisfies the restricted secant inequality (RSI) condition and  the gradients of each local cost function at optimal points are the same,  which are also stronger than the P--{\L} condition.

In many applications, such as optimal power flow problems \cite{guo2016case}, resource allocation problems \cite{tychogiorgos2013non}, and empirical risk minimization problems \cite{bottou2018optimization}, the cost functions are usually nonconvex. Thus, distributed nonconvex optimization has gained considerable attentions, e.g., \cite{bianchi2013performance,wai2015consensus,xu2017newton,
tatarenko2017non,zeng2018nonconvex,hong2017prox,sun2018distributed,sun2019distributed,
hajinezhad2019perturbed,pmlr-v80-hong18a,daneshmand2018second,swenson2019distributed,
vlaski2019distributeda,vlaski2019distributedb}. In these studies, basic convergence results typically ensure that distributed algorithms converge to stationary points. 
For example, in \cite{tatarenko2017non,hong2017prox,daneshmand2018second,sun2018distributed,
hajinezhad2019perturbed,sun2019distributed}, it was shown that the first-order stationary point can be found with an $\mathcal{O}(1/T)$ convergence rate when each local cost function is smooth, where $T$ is the total number of iterations; in  \cite{xu2017newton,pmlr-v80-hong18a,daneshmand2018second,swenson2019distributed,vlaski2019distributedb}, it was shown that the second-order stationary (SoS) point can be found under additional assumptions, such as the Kurdyka--{\L}ojasiewicz (K--{\L}) condition, Lipschitz-continuous Hessian, or a suitably chosen initialization.

\begin{table*}[htbp]
\caption{Comparison of this paper to some related algorithms obtaining linear convergence for distributed optimization.}
\label{nonconvex:table-linear}
\vskip -0.3in
\begin{center}
\begin{scriptsize}
\scalebox{1}{
\begin{tabular}{M{1.1cm}|M{7.05cm}|M{4.5cm}|M{2.1cm}N}
\hline
Reference&Cost function&Communication strategy&Communication type&\\[14pt]

\hline
\cite{lu2012zero}&Strongly convex local cost functions with locally Lipschitz Hessian&Connected undirected, one variable&Continuous-time&\\[4pt]

\hline
\cite{kia2015distributed}&Strongly convex and smooth local cost functions&Connected undirected, one variable&Event-triggered&\\[4pt]

\hline
\multirow{1}{*}[-0pt]{\cite{shi2014linear,ling2015dlm,jakovetic2015linear,makhdoumi2017convergence,
berahas2018balancing}}&Strongly convex and smooth local cost functions&Connected undirected, one variable&Discrete-time&\\[4pt]

\hline
\cite{mokhtari2016dqm}&Strongly convex and smooth local cost functions with Lipschitz Hessian&Connected undirected, one variable&Discrete-time&\\[14pt]

\hline
\cite{nedic2017achieving}&Strongly convex and smooth local cost functions&Uniformly jointly strongly connected, two variables&Discrete-time&\\[14pt]

\hline
\cite{qu2017accelerated}&Strongly convex and smooth local cost functions&Connected undirected, three variables&Discrete-time&\\[4pt]

\hline
\multirow{1}{*}[-0pt]{\cite{qu2018harnessing,jakovetic2019unification,mansoori2019flexible}}&Strongly convex and smooth local cost functions&Connected undirected, two variables&Discrete-time&\\[4pt]

\hline
\cite{xi2018add}&Strongly convex and smooth local cost functions&Strongly connected, three variables&Discrete-time&\\[4pt]

\hline
\cite{xu2018convergence}&Strongly convex and smooth local cost functions&Undirected stochastic graphs with random failures, two variables&Discrete-time&\\[14pt]

\hline
\cite{Damiano-TAC2016}&Convex and smooth local cost functions, strongly convex global cost function&Connected undirected, four variables&Discrete-time&\\[14pt]


\hline
\cite{maros2019q}&Convex and smooth local cost functions, strongly convex global cost function&Connected undirected, one variable&Discrete-time&\\[14pt]

\hline
\cite{tian2018asy}&Nonconvex and smooth local cost functions, strongly convex global cost function&Uniformly jointly strongly connected with delays, five variables&Discrete-time&\\[14pt]

\hline
\cite{shi2015extra}&Convex and smooth local cost functions, restricted strongly convex global cost function, unique optimum&Connected undirected, one variable&Discrete-time&\\[14pt]

\hline
\cite{zeng2017extrapush,xi2017dextra}&Convex and smooth local cost functions, restricted strongly convex global cost function, unique optimum&Strongly connected, two variables&Discrete-time&\\[14pt]

\hline
\cite{Yi2018distributed}&Convex and smooth local cost functions, restricted strongly convex global cost function, unique optimum&Connected undirected, one variable&Event-triggered&\\[14pt]

\hline
\cite{liang2019exponential}&Convex and  smooth local cost functions,  the primal--dual gradient map is metric subregularity&Connected undirected, two variables&Discrete-time&\\[14pt]

\hline
\cite{yi2019exponential}&Smooth local cost functions, the global cost function satisfies the RSI condition,  the gradients of each local cost function at optimal points are the same&Connected undirected, one variable&Discrete-time&\\[14pt]

\hline
This paper&Smooth local cost functions, the global cost function satisfies the P--{\L} condition&Connected undirected, one variable&Discrete-time&\\[14pt]

\hline
\end{tabular}
}
\end{scriptsize}
\end{center}
\vskip 0in
\end{table*}

Note that aforementioned distributed optimization algorithms use at least gradient information of the cost functions, and sometime even second- or higher-order information. However, in some practical applications, explicit expressions of the gradients are often unavailable or difficult to obtain \cite{conn2009introduction}. For example, the cost functions of many big data problems that deal with complex data generating processes cannot be explicitly defined \cite{chen2017zoo}. Thus, zeroth-order (gradient-free) optimization algorithms are needed. A key step in zeroth-order optimization algorithms is to estimate the gradient of the cost function by sampling the cost function values. Various gradient estimation methods have been developed, e.g., \cite{nesterov2017random,agarwal2010optimal}.
Some recent works have combined these gradient estimation methods with distributed first-order algorithms. For instance, the authors of \cite{yuan2014randomized,yuan2015gradient,sahu2018distributed,wang2019distributed,
yu2019distributed} and \cite{hajinezhad2019zone,tang2019distributed} proposed distributed zeroth-order algorithms for distributed convex and nonconvex optimization, respectively.

The main contribution of this paper is on solving distributed nonconvex optimization problems. We first propose a distributed first-order primal--dual algorithm. When each local cost function is smooth, we show that it finds the first-order stationary point with a rate $\mathcal{O}(1/T )$ and that the cost difference between the global optimum and the resulting stationary point is bounded. We also show that not only the proposed algorithm can find a global optimum but also the convergence rate is linear  under an additional assumption that the global cost function satisfies the P--{\L} condition. This condition is weaker than the (restrict) strong convexity condition assumed in \cite{lu2012zero,kia2015distributed,shi2014linear,ling2015dlm,shi2015extra,jakovetic2015linear,
mokhtari2016dqm,makhdoumi2017convergence,Damiano-TAC2016,nedic2017achieving,
zeng2017extrapush,xi2017dextra,qu2018harnessing,mansoori2019flexible,
qu2017accelerated,xi2018add,xu2018convergence,Yi2018distributed,
tian2018asy,yi2019exponential,maros2019q,jakovetic2019unification,berahas2018balancing} since it does not require convexity and the global minimizer is not necessarily unique. This condition is also different from the metric subregularity criterion assumed in \cite{liang2019exponential}. In other words, we show that for a larger class of cost functions than strongly convex functions, the global optimum can be founded linearly by the proposed distributed algorithm. It should be highlighted that the P--{\L} constant is not used to design the algorithm parameters. Noting that generally the P--{\L} condition is difficult to check, with the above property, the P--{\L} condition does not need to be checked when implementing the considered algorithm, which is a significant innovation.
Another innovation is that the proofs of both sublinear and linear convergence are based on the same appropriately designed Lyapunov function, which facilitates extending our results to other settings, such as event-triggered communication. 
TABLE~\ref{nonconvex:table-linear} compares  this paper with other algorithms that obtain linear convergence for distributed optimization. TABLE~\ref{nonconvex:table-nonconvex} summarizes the comparison on distributed nonconvex optimization.


\begin{table*}[!ht]
\caption{Comparison of this paper to some related distributed nonconvex optimization algorithms.}
\label{nonconvex:table-nonconvex}
\vskip -3in
\begin{center}
\begin{scriptsize}
\scalebox{1}{
\begin{tabular}{M{1.1cm}|M{6.8cm}|M{4.5cm}|M{2.3cm}N}
\hline
Reference&Cost function&Communication strategy&Convergence rate&\\[4pt]

\hline
\cite{tatarenko2017non}&Lipschitz and smooth local cost functions, the set of stationary points is a union of finitely many connected components, no saddle points&Uniformly jointly strongly connected, two variables&$\mathcal{O}(1/T)$ to a local optimum&\\[24pt]

\hline
\cite{zeng2018nonconvex}&Lipschitz local cost functions&Connected undirected, one variable&Asymptotic&\\[4pt]

\hline
\multirow{1}{*}[-0pt]{\cite{hong2017prox,sun2018distributed,sun2019distributed,hajinezhad2019perturbed}}&Smooth local cost functions&Connected undirected, one variable&$\mathcal{O}(1/T)$&\\[4pt]

\hline
\cite{pmlr-v80-hong18a}&Smooth local cost functions with  Lipschitz Hessian,  the global cost function satisfies the K--{\L} condition, $p=1$&Connected undirected, one variable&Almost surely to an SoS solution&\\[14pt]

\hline
\multirow{2}{*}[-14pt]{\cite{daneshmand2018second}}&Smooth local cost functions&Strongly connected, two variables&$\mathcal{O}(1/T)$&\\[4pt]
\cline{2-4}
&Smooth local cost functions, a suitably chosen initialization&Connected undirected or strongly connected with $p=1$, suitably chosen weight matrices, two variables&Almost surely to an SoS solution&\\[24pt]

\hline
\multirow{2}{*}[-7pt]{\parbox{1.1cm}{\centering This paper}}&Smooth local cost functions&\multirow{2}{*}[-2pt]{\parbox{4.5cm}{\centering Connected undirected, one variable}}&$\mathcal{O}(1/T)$&\\[4pt]
\cline{2-2}\cline{4-4}
&The global cost function also satisfies the P--{\L} condition&&Linearly to a global optimum&\\[14pt]

\hline
\end{tabular}
}
\end{scriptsize}
\end{center}
\vskip -0.1in
\end{table*}


\begin{table*}[t]
\caption{Comparison of this paper to some related distributed zeroth-order optimization algorithms.}
\label{nonconvex:table-zero}
\vskip -3in
\begin{center}
\begin{scriptsize}
\scalebox{1}{
\begin{tabular}{M{1.1cm}|M{2.3cm}|M{1.2cm}|M{2.55cm}|M{6.95cm}N}
\hline
Reference&Cost function&Sampling type&Communication  strategy&Convergence result&\\[14pt]

\hline
\cite{yuan2014randomized}&Convex, Lipschitz&Two noise-free samplings&Uniformly jointly strongly connected, one variable&$\lim_{T\to\infty}\mathbf{E}[f(x_{i,T})-f^*]=0$&\\[24pt]

\hline
\cite{yuan2015gradient}&Convex, Lipschitz&Two noise-free samplings&Uniformly jointly strongly connected, two variables&$\mathbf{E}[f(x_{i,T})-f^*]=\mathcal{O}(p^3n^2\ln(T)/\sqrt{T})$&\\[24pt]

\hline
\cite{sahu2018distributed}&Strongly convex, smooth&$2p$ noise samplings&Connected undirected in average, one variable&$\mathbf{E}[\sum_{i=1}^{n}\|x_{i,T}-x^*\|^2/n]=\mathcal{O}(pn^2/\sqrt{T})$&\\[24pt]

\hline
\cite{wang2019distributed}&Strictly convex, Lipschitz&Two noise samplings&Uniformly jointly complete, one variable&$\mathbf{E}[\sum_{i=1}^{n}\|x_{i,T}-x^*\|^2/n]=\mathcal{O}(1/\sqrt{T})
+\mathcal{O}(p^2/T)$&\\[24pt]


\hline
\multirow{2}{*}[-8pt]{\cite{yu2019distributed}}&Convex, Lipschitz&\multirow{2}{*}[-4pt]{\parbox{1.2cm}{\centering Two noise-free samplings}}&\multirow{2}{*}[-4pt]{\parbox{2.55cm}{\centering Uniformly jointly strongly connected, one variable}}&$\mathbf{E}[f(x_{i,T})-f^*]=\mathcal{O}(p\sqrt{n/T})$&\\[6pt]
\cline{2-2}\cline{5-5}
&Strongly convex,  Lipschitz&&&$\mathbf{E}[f(x_{i,T})-f^*]=\mathcal{O}(p^2n^2\ln(T)/T)$&\\[14pt]

\hline
\cite{hajinezhad2019zone}&Nonconvex, Lipschitz, smooth&$\mathcal{O}(T)$ stochastic samplings&Connected undirected, one variable&$\sum_{k=1}^{T}\mathbf{E}[\sum_{i=1}^{n}\|x_{i,k}-\bar{x}_k\|^2+\|\nabla f_i(x_{i,k})\|^2]/n
=\mathcal{O}(p^2n)$&\\[24pt]

\hline
\multirow{4}{*}[-24pt]{\cite{tang2019distributed}}&Nonconvex, Lipschitz, smooth&\multirow{2}{*}[-8pt]{\parbox{1.2cm}{\centering Two noise-free samplings}}&\multirow{2}{*}[-6pt]{\parbox{2.55cm}{\centering Connected undirected, two variables}}&$~~~~~~~~~$$\mathbf{E}[\sum_{i=1}^{n}\|x_{i,T}-\bar{x}_T\|/n]=\mathcal{O}(1/T)$, $~~~~~~~~~$ $\sum_{k=1}^{T}\mathbf{E}[\|\nabla f(\bar{x}_k)\|^2]=\mathcal{O}(\sqrt{pT})$&\\[14pt]
\cline{2-2}\cline{5-5}
&Nonconvex,  smooth,  P--{\L} condition&&&$~~~~~~~~~$$\mathbf{E}[\sum_{i=1}^{n}\|x_{i,T}-\bar{x}_T\|/n]=\mathcal{O}(p/T^2)$, $~~~~~~~~~$
$\mathbf{E}[f(\bar{x}_T)-f^*]=\mathcal{O}(p/T)$ (the P--{\L} constant is used)&\\[24pt]
\cline{2-5}
&Nonconvex, smooth&\multirow{2}{*}[-4pt]{\parbox{1.2cm}{\centering $2p$ noise-free samplings}}&\multirow{2}{*}[-4pt]{\parbox{2.55cm}{\centering Connected undirected, three variables}}&$\sum_{k=1}^{T}[\sum_{i=1}^{n}\|x_{i,k}-\bar{x}_k\|^2/n+\|\nabla f(\bar{x}_k)\|^2]
=\mathcal{O}(1)$&\\[6pt]
\cline{2-2}\cline{5-5}
&P--{\L} condition in addition&&&Linearly to a global optimum (the P--{\L} constant is used)&\\[14pt]

\hline
\multirow{2}{*}[-6pt]{\parbox{1.1cm}{\centering This paper}}&Nonconvex, smooth&\multirow{2}{*}[-2pt]{\parbox{1.2cm}{\centering $p+1$ noise-free samplings}}&\multirow{2}{*}[-6pt]{\parbox{2.55cm}{\centering Connected undirected, one variable}}&$\sum_{k=1}^{T}[\sum_{i=1}^{n}\|x_{i,k}-\bar{x}_k\|^2/n+\|\nabla f(\bar{x}_k)\|^2]
=\mathcal{O}(1)$&\\[6pt]
\cline{2-2}\cline{5-5}
&P--{\L} condition in addition&&&Linearly to a global optimum&\\[14pt]

\hline
\end{tabular}
}
\end{scriptsize}
\end{center}
\vskip -0.1in
\end{table*}

Motivated by the situations where the gradient information is unavailable, we then propose a distributed zeroth-order algorithm, by integrating the proposed distributed first-order algorithm with the deterministic gradient estimator proposed in \cite{agarwal2010optimal}. We show that it has the same convergence properties as the proposed first-order algorithm under the same conditions. It should be mentioned that the analysis of both sublinear and linear convergence for our zeroth-order algorithm is based on the Lyapunov function modified from the Lyapunov function for the first-order algorithm. Compared with \cite{tang2019distributed}, which also proposed a distributed deterministic zeroth-order algorithm and established the same convergence properties under the same conditions as ours, one  innovation of our zeroth-order algorithm is that the P--{\L} constant, which is normally difficult to determine, is not used for designing the algorithm.
Moreover, the proposed zeroth-order algorithm only requires each agent to communicate one $p$-dimensional variable with its neighbors at each iteration, where $p$ is the dimension of the decision variable, while the algorithm proposed in \cite{tang2019distributed} requires each agent to communicate three $p$-dimensional variables. The comparison between this paper and the literature on distributed zeroth-order optimization is summarized in TABLE~\ref{nonconvex:table-zero}.

The rest of this paper is organized as follows.
Section~\ref{nonconvex:sec-preliminary} introduces some preliminaries. Section~\ref{nonconvex:sec-problem} presents the problem formulation and assumptions.
Sections~\ref{nonconvex:sec-main-dc} and \ref{nonconvex:sec-main-zo} provide the distributed first- and zeroth-order primal--dual algorithms, and analyze their convergence properties, respectively. Simulations are given in Section~\ref{nonconvex:sec-simulation}. Finally, concluding remarks are offered in Section~\ref{nonconvex:sec-conclusion}. To improve the readability, all the proofs are given in the appendix.
A preliminary conference version of this paper has been presented as \cite{yi2020linearconv}. That paper only discusses the first-order algorithm under the P--{\L} condition. The current paper is thus a significant extension as we consider both first- and zeroth-order algorithms, and analyse their convergence properties with and without the P--{\L} condition, and provide detailed proofs as well as comparisons to related algorithms in the literature.

\noindent {\bf Notations}: $\mathbb{N}_0$ and $\mathbb{N}_+$ denote the set of nonnegative and positive integers, respectively. $\{{\bf e}_1,\dots,{\bf e}_p\}$ represents the standard basis of $\mathbb{R}^p$. $[n]$ denotes the set $\{1,\dots,n\}$ for any positive constant integer $n$. $\col(z_1,\dots,z_k)$ is the concatenated column vector of vectors $z_i\in\mathbb{R}^{p_i},~i\in[k]$.
${\bf 1}_n$ (${\bf 0}_n$) denotes the column one (zero) vector of dimension $n$. ${\bf I}_n$ is the $n$-dimensional identity matrix. Given a vector $[x_1,\dots,x_n]^\top\in\mathbb{R}^n$, $\diag([x_1,\dots,x_n])$ is a diagonal matrix with the $i$-th diagonal element being $x_i$. The notation $A\otimes B$ denotes the Kronecker product
of matrices $A$ and $B$. $\nullrank(A)$ is the null space of matrix $A$.
Given two symmetric matrices $M,N$, $M\ge N$ means that $M-N$ is positive semi-definite. $\rho(\cdot)$ stands for the spectral radius for matrices and $\rho_2(\cdot)$ indicates the minimum
positive eigenvalue for matrices having positive eigenvalues. $\|\cdot\|$ represents the Euclidean norm for
vectors or the induced 2-norm for matrices. For any square matrix $A$, denote $\|x\|_A^2$=$x^\top Ax$.
Given a differentiable function $f$, $\nabla f$ denotes the gradient of $f$.

\section{Preliminaries}\label{nonconvex:sec-preliminary}
In this section, we present some definitions and properties related to algebraic graph theory, smooth functions, the P--{\L} condition, and the deterministic gradient
estimator.

\subsection{Algebraic Graph Theory}

Let $\mathcal G=(\mathcal V,\mathcal E, A)$ denote a weighted undirected graph with the set of vertices (nodes) $\mathcal V =[n]$, the set of links (edges) $\mathcal E
\subseteq \mathcal V \times \mathcal V$, and the weighted adjacency matrix
$A =A^{\top}=(a_{ij})$ with nonnegative elements $a_{ij}$. A link of $\mathcal G$ is denoted by $(i,j)\in \mathcal E$ if $a_{ij}>0$, i.e., if vertices $i$ and $j$ can communicate with each other. It is assumed that $a_{ii}=0$ for all $i\in [n]$. Let $\mathcal{N}_i=\{j\in [n]:~ a_{ij}>0\}$ and $\deg_i=\sum\limits_{j=1}^{n}a_{ij}$ denote the neighbor set and weighted degree of vertex $i$, respectively. The degree matrix of graph $\mathcal G$ is $\Deg=\diag([\deg_1, \cdots, \deg_n])$. The Laplacian matrix is $L=(L_{ij})=\Deg-A$. A  path of length $k$ between vertices $i$ and $j$ is a subgraph with distinct vertices $i_0=i,\dots,i_k=j\in [n]$ and edges $(i_j,i_{j+1})\in\mathcal E,~j=0,\dots,k-1$.
An undirected graph is  connected if there exists at least one path between any two distinct vertices. If the graph $\calG$ is connected, then its Laplacian matrix $L$ is positive semi-definite and $\nullrank(L)=\{{\bf 1}_n\}$, see \cite{mesbahi2010graph}.

\subsection{Smooth Function}
\begin{definition}
The function $f(x):~\mathbb{R}^p\mapsto\mathbb{R}$ is smooth with constant $L_f>0$ if it is differentiable and
\begin{align}\label{nonconvex:smooth}
\|\nabla f(x)-\nabla f(y)\|\le L_{f}\|x-y\|,~\forall x,y\in \mathbb{R}^p.
\end{align}
\end{definition}
From Lemma~1.2.3 in \cite{nesterov2018lectures}, we know that \eqref{nonconvex:smooth} implies
\begin{align}
|f(y)-f(x)-(y-x)^\top\nabla f(x)|
\le\frac{L_f}{2}\|y-x\|^2,~\forall x,y\in\mathbb{R}^{p}. \label{nonconvex:lemma:smooth}
\end{align}

\subsection{Polyak--{\L}ojasiewicz Condition}
Let $f(x):~\mathbb{R}^p\mapsto\mathbb{R}$ be a differentiable function. Let $\mathbb{X}^*=\argmin_{x\in\mathbb{R}^p}f(x)$ and $f^*=\min_{x\in\mathbb{R}^p}f(x)$. Moreover, we assume that $f^*>-\infty$.
\begin{definition} 
The function $f$ satisfies the Polyak--{\L}ojasiewicz (P--{\L}) condition with constant  $\nu>0$ if
\begin{align}
\frac{1}{2}\|\nabla f(x)\|^2\ge \nu( f(x)-f^*),~\forall x\in \mathbb{R}^p.\label{nonconvex:equ:plc}
\end{align}
\end{definition}
It is straightforward to see that every (essentially or weakly) strongly convex function satisfies the P--{\L} condition.
The P--{\L} condition implies that every stationary point is a global minimizer, i.e., $\mathbb{X}^*=\{x\in\mathbb{R}^p:\nabla f(x)={\bf 0}_p\}$. But unlike the (essentially or weakly) strong convexity, the P--{\L} condition alone does not imply convexity of $f$. Moreover, it does not imply that $\mathbb{X}^*$ is a singleton either. The function $f(x)=x^2+3\sin^2(x)$ is an example of a nonconvex function satisfying the P--{\L} condition with $\nu=1/32$, see \cite{karimi2016linear}. More examples of nonconvex functions which satisfy the P--{\L} condition can be found in \cite{karimi2016linear,zhang2015restricted}.

Although it is difficult to precisely characterize the general class of functions for which the P--{\L} condition is satisfied, in \cite{karimi2016linear}, one special case was given as follows:
\begin{lemma}
Let $f(x)=g(Ax)$, where $g:\mathbb{R}^p\mapsto\mathbb{R}$ is a strongly convex function and $A\in\mathbb{R}^{p\times p}$ is a matrix, then $f$ satisfies the P--{\L} condition.
\end{lemma}
Moreover, from Theorem~2 in \cite{karimi2016linear}, we know that the following property holds.
\begin{lemma}\label{nonconvex:lemma:plc}
Suppose that the function $f$ satisfies the P--{\L} condition. Let $\calP_{\mathbb{X}^*}(x)$ be the projection of $x$ onto the set $\mathbb{X}^*$, i.e., $\calP_{\mathbb{X}^*}(x)=\argmin_{y\in \mathbb{X}^*}\|x-y\|^2$. Suppose that  $\calP_{\mathbb{X}^*}(x),~\forall x\in\mathbb{R}^p$ is well defined. Then,
\begin{align}\label{nonconvex:lemma:plc-equ}
f(x)-f^*\ge2\nu\|\calP_{\mathbb{X}^*}(x)-x\|^2,~\forall x\in \mathbb{R}^p.
\end{align}
\end{lemma}

From Theorem~2.2.11 in \cite{nesterov2018lectures}, we know that $\calP_{\mathbb{X}^*}(\cdot)$ is well defined if $\mathbb{X}^*$ is closed and convex.

\subsection{Deterministic Gradient Estimator}
Let $f(x):~\mathbb{R}^p\mapsto\mathbb{R}$ be a differentiable function. The authors of \cite{agarwal2010optimal} proposed the following deterministic gradient estimator:
\begin{align}
\hat{\nabla}f(x,\delta)&=\frac{1}{\delta}\sum_{l=1}^{p}(f(x+\delta{\bf e}_l)-f(x)){\bf e}_l,\label{zero:determin-gradient}
\end{align}
where $\delta>0$ is an exploration parameter. This gradient estimator can be calculated by sampling the function values of $f$ at $p+1$ points. From equation (16) in \cite{agarwal2010optimal}, we know that $\hat{\nabla}f(x,\delta)$ is close to $\nabla f(x)$ when $\delta$ is small, which is summarized in the following lemma.
\begin{lemma}\label{zero:gradient-est}
Suppose that $f$ is smooth with constant $L_f$, then
\begin{align}\label{zero:gradient-close}
\|\hat{\nabla}f(x,\delta)-\nabla f(x)\|\le\frac{\sqrt{p}L_f\delta}{2},~\forall x\in\mathbb{R}^p,~\forall \delta>0.
\end{align}
\end{lemma}

\section{Problem Formulation and assumptions}\label{nonconvex:sec-problem}
Consider a network of $n$ agents, each of which has a private local cost function $f_i: \mathbb{R}^{p}\rightarrow \mathbb{R}$.
All agents collaborate to solve the following optimization problem
\begin{align}\label{nonconvex:eqn:xopt}
 \min_{x\in \mathbb{R}^p} f(x)=\frac{1}{n}\sum_{i=1}^nf_i(x).
\end{align}
The communication among agents is described by a weighted undirected graph $\mathcal{G}$.
Let $\mathbb{X}^*$ and $f^*$ denote the optimal set and the minimum function value of the optimization problem \eqref{nonconvex:eqn:xopt}, respectively.  The following assumptions are made.

\begin{assumption}\label{nonconvex:ass:graph}
The undirected graph $\mathcal G$ is connected.
\end{assumption}

\begin{assumption}\label{nonconvex:ass:optset}
The optimal set $\mathbb{X}^*$ is nonempty and $f^*>-\infty$.
\end{assumption}

\begin{assumption}\label{nonconvex:ass:fiu}
Each  local cost function $f_i(x)$ is smooth with constant $L_{f}>0$.
\end{assumption}

\begin{assumption}\label{nonconvex:ass:fil} The global cost function $f(x)$ satisfies the P--{\L} condition with constant $\nu>0$.
\end{assumption}

\begin{remark}
Assumptions~\ref{nonconvex:ass:graph}--\ref{nonconvex:ass:fiu} are common in the literature, e.g., \cite{shi2014linear,shi2015extra}.
Assumption~\ref{nonconvex:ass:fil}  is  weaker than the assumption that the global or each local cost function is strongly convex. It should be highlighted that the convexity of the cost functions and the boundedness of their gradients are not assumed. Moreover, we do not assume that $\mathbb{X}^*$ is a singleton or finite set either.
\end{remark}


\section{Distributed first-order primal--dual algorithm}\label{nonconvex:sec-main-dc}
In this section, we propose a distributed first-order primal--dual algorithm and analyze its convergence rate.

\subsection{Algorithm Description}
In this section, we present the derivation of our proposed algorithm.

Denote $\bsx=\col(x_1,\dots,x_n)$, $\tilde{f}(\bsx)=\sum_{i=1}^{n}f_i(x_i)$,  and $\bsL=L\otimes {\bf I}_p$.
Recall that the Laplacian matrix $L$ is positive semi-definite and $\nullrank(L)=\{{\bf 1}_n\}$ when $\calG$ is connected. The optimization problem \eqref{nonconvex:eqn:xopt} is equivalent to the following constrained optimization problem:
\begin{mini}
{\bsx\in \mathbb{R}^{np}}{\tilde{f}(\bsx)}{\label{nonconvex:eqn:xoptcon}}{}
\addConstraint{\bsL^{1/2}\bsx=}{~{\bf 0}_{np}.}{}
\end{mini}
Here, $\bsL^{1/2}=L^{1/2}\otimes {\bf I}_p$ and $L^{1/2}$ is the square root of the positive semi-definite matrix $L$. Moreover, we use $\bsL^{1/2}\bsx={\bf 0}_{np}$ rather than $\bsL\bsx={\bf 0}_{np}$ as the constraint since they are both equivalent to $\bsx={\bf 1}_n\otimes x$  due to the fact that $\nullrank(L^{1/2})=\nullrank(L)=\{{\bf 1}_n\}$, but the first has a particular property which will be discussed in Remark~\ref{nonconvex:remark3}.

Let $\bsu\in\mathbb{R}^{np}$ denote the dual variable. Then the augmented Lagrangian function associated with \eqref{nonconvex:eqn:xoptcon} is
\begin{align}\label{nonconvex:lagran}
\calA(\bsx,\bsu)=\tilde{f}(\bsx)+\frac{\alpha}{2}\bsx^\top\bsL\bsx+\beta\bsu^\top\bsL^{1/2}\bsx,
\end{align}
where $\alpha>0$ and $\beta>0$ are  the regularization parameters.

Based on the primal--dual gradient method, a distributed first-order algorithm to solve \eqref{nonconvex:eqn:xoptcon} is
\begin{subequations}\label{nonconvex:kiau-algo-compact}
\begin{align}
\bm{x}_{k+1}&=\bm{x}_k-\eta(\alpha\bsL\bm{x}_k+\beta\bsL^{1/2}\bm{u}_k+\nabla \tilde{f}(\bm{x}_k)),\\
\bm{u}_{k+1}&=\bm{u}_{k}+\eta\beta\bsL^{1/2}\bm{x}_k,~\forall \bsx_0,~\bsu_0\in\mathbb{R}^{np},
\end{align}
\end{subequations}
where $\eta>0$ is a fixed stepsize.
Denote $\bsv_k=\col(v_{1,k},\dots,v_{n,k})=\bsL^{1/2}\bm{u}_k$. Then the  recursion \eqref{nonconvex:kiau-algo-compact} can be rewritten as
\begin{subequations}\label{nonconvex:kia-algo-dc-compact}
\begin{align}
\bm{x}_{k+1}&=\bm{x}_k-\eta(\alpha\bsL\bm{x}_k+\beta\bm{v}_k+\nabla \tilde{f}(\bm{x}_k)),\label{nonconvex:kia-algo-dc-compact-x}\\
\bm{v}_{k+1}&=\bm{v}_k+\eta\beta\bsL\bm{x}_k,~\forall \bsx_0\in\mathbb{R}^{np},~\sum_{j=1}^nv_{j,0}={\bf 0}_p.\label{nonconvex:kia-algo-dc-compact-v}
\end{align}
\end{subequations}
The initialization condition $\sum_{j=1}^nv_{j,0}={\bf 0}_p$ is derived from $\bsv_0=\bsL^{1/2}\bm{u}_0$, and it is easy to be satisfied, for example, $v_{i,0}={\bf0}_p,~\forall i\in[n]$, or $v_{i,0}=\sum_{j=1}^{n}L_{ij}x_{j,0},~\forall i\in[n]$.
It is straightforward to verify that the algorithm \eqref{nonconvex:kia-algo-dc-compact} is equivalent to the EXTRA algorithm proposed in \cite{shi2015extra} with mixing matrices $\bsW={\bf I}_{np}-\eta\alpha\bsL$ and $\tilde{\bsW}={\bf I}_{np}-\eta\alpha\bsL+\eta^2\beta^2\bsL$. Note that \eqref{nonconvex:kia-algo-dc-compact} can be written agent-wise as
\begin{subequations}\label{nonconvex:kia-algo-dc}
\begin{align}
x_{i,k+1} &= x_{i,k}-\eta(\alpha\sum_{j=1}^nL_{ij}x_{j,k}+\beta v_{i,k}+\nabla f_i(x_{i,k})), \label{nonconvex:kia-algo-dc-x}\\
v_{i,k+1} &=v_{i,k}+ \eta\beta\sum_{j=1}^n L_{ij}x_{j,k},~ x_{i,0}\in\mathbb{R}^p, ~\sum_{j=1}^nv_{j,0}={\bf 0}_p,~
\forall i\in[n].  \label{nonconvex:kia-algo-dc-q}
\end{align}
\end{subequations}
This corresponds to our proposed distributed first-order primal--dual algorithm, which is presented in pseudo-code as Algorithm~\ref{nonconvex:algorithm-pdgd}.
\begin{algorithm}[tb]
\caption{Distributed First-Order Primal--Dual Algorithm}
\label{nonconvex:algorithm-pdgd}
\begin{algorithmic}[1]
\STATE \textbf{Input}: parameters $\alpha>0$, $\beta>0$, and $\eta>0$.
\STATE \textbf{Initialize}: $ x_{i,0}\in\mathbb{R}^p$ and $v_{i,0}={\bf 0}_p,~
\forall i\in[n]$.
\FOR{$k=0,1,\dots$}
\FOR{$i=1,\dots,n$  in parallel}
\STATE  Broadcast $x_{i,k}$ to $\mathcal{N}_i$ and receive $x_{j,k}$ from $j\in\mathcal{N}_i$;
\STATE  Update $x_{i,k+1}$ by \eqref{nonconvex:kia-algo-dc-x};
\STATE  Update $v_{i,k+1}$ by \eqref{nonconvex:kia-algo-dc-q}.
\ENDFOR
\ENDFOR
\STATE  \textbf{Output}: $\{\bsx_{k}\}$.
\end{algorithmic}
\end{algorithm}

\begin{remark}\label{nonconvex:remark:alg}
In the literature, various distributed first-order algorithms have been proposed to solve the nonconvex optimization problem \eqref{nonconvex:eqn:xopt}, for example, distributed gradient descent algorithm  \cite{zeng2018nonconvex,daneshmand2018second}, distributed gradient tracking algorithm \cite{daneshmand2018second}, distributed algorithm based on a novel appro\underline{x}imate
\underline{filt}ering-then-pr\underline{e}dict and t\underline{r}acking (xFILTER) strategy \cite{sun2019distributed}. Compared with the proposed distributed algorithm \eqref{nonconvex:kia-algo-dc}, these algorithms have some potential drawbacks. For the distributed gradient descent algorithm, existing studies, such as \cite{zeng2018nonconvex,daneshmand2018second}, only showed that the output of the algorithm converges to a neighborhood of a stationary point unless additional assumptions, such as the boundedness of the gradients of cost functions, are assumed. In the distributed gradient tracking algorithm \cite{daneshmand2018second}, at each iteration each agent $i$ needs to communicate one additional $p$-dimensional variables besides the communication of $x_{i,k}$ with its neighbors. The xFILTER algorithm \cite{sun2019distributed} is a double-loop algorithm and thus more complicated than \eqref{nonconvex:kia-algo-dc}.
\end{remark}

\subsection{Convergence Analysis}
In this section, we provide convergence analysis for both without and with Assumption~\ref{nonconvex:ass:fil}.

Denote $K_n={\bf I}_n-\frac{1}{n}{\bf 1}_n{\bf 1}^{\top}_n$, $\bsK=K_n\otimes {\bf I}_p$, $\bsH=\frac{1}{n}({\bf 1}_n{\bf 1}_n^\top\otimes{\bf I}_p)$, $\bar{x}_k=\frac{1}{n}({\bf 1}_n^\top\otimes{\bf I}_p)\bsx_k$, $\bar{\bsx}_k={\bf 1}_n\otimes\bar{x}_k$, $\bsg_k=\nabla\tilde{f}(\bsx_k)$, $\bar{\bsg}_k=\bsH\bsg_{k}$, $\bsg^0_k=\nabla\tilde{f}(\bar{\bsx}_k)$, $\bar{\bsg}_k^0=\bsH\bsg^0_{k}={\bf 1}_n\otimes\nabla f(\bar{x}_k)$, and
\begin{align*}
\hat{V}_k&=\|\bm{x}_k\|^2_{\bsK}+\|\bsv_k
+\frac{1}{\beta}\bsg_k^0\|^2_{\bsK}+n(f(\bar{x}_k)-f^*),\\
W_k&=\|\bsx_{k}-\bar{\bsx}_k\|^2
+\|\bm{v}_k+\frac{1}{\beta}\bsg_k^0\|^2_{\bsK}+\|\bar{\bsg}_{k}\|^2
+\|\bar{\bsg}_{k}^0\|^2.
\end{align*}

We have the following convergence result for Algorithm~\ref{nonconvex:algorithm-pdgd} without Assumption~\ref{nonconvex:ass:fil}.
\begin{theorem}\label{nonconvex:thm-sm}
Suppose that Assumptions~\ref{nonconvex:ass:graph}--\ref{nonconvex:ass:fiu} hold.  Let $\{\bsx_k\}$ be the sequence generated by Algorithm~\ref{nonconvex:algorithm-pdgd} with $\alpha\in(\beta+\kappa_1,\kappa_2\beta]$, $\beta>\max\{\frac{\kappa_1}{\kappa_2-1},~\kappa_3,~\kappa_4\}$, and $\eta\in(0,\min\{\frac{\epsilon_1}{\epsilon_2},~\frac{\epsilon_3}{\epsilon_4},
~\frac{\epsilon_5}{\epsilon_6}\})$. Then,
\begin{align}
&\frac{\sum_{k=0}^{T}W_k}{T+1}\le \frac{\epsilon_8\hat{V}_{0}}{\epsilon_7(T+1)},~\forall T\in\mathbb{N}_0,\label{nonconvex:thm-sm-equ1}\\
&f(\bar{x}_{T+1})-f^*\le\frac{\epsilon_8\hat{V}_{0}}{n},~\forall T\in\mathbb{N}_0,\label{nonconvex:thm-sm-equ2}
\end{align}
where
\begin{align*}
\kappa_1&=\frac{1}{2\rho_2(L)}(2+3L_f^2),~
\kappa_2>1,\\
\kappa_3&=\frac{1}{4}(1+(1+8\kappa_2+\frac{8}{\rho_2(L)})^{\frac{1}{2}}),\\
\kappa_4&=(\kappa_2+\frac{1}{\rho_2(L)})L_f^2
+((\kappa_2+\frac{1}{\rho_2(L)})^2L_f^2+2)^{\frac{1}{2}}L_f,\\
\epsilon_1&=(\alpha-\beta)\rho_2(L)-\frac{1}{2}(2+3L_f^2),\\
\epsilon_2&=\beta^2\rho(L)+(2\alpha^2+\beta^2)\rho^2(L)+\frac{5}{2}L_f^2,\\
\epsilon_3&=\beta-\frac{1}{2}-\frac{\alpha}{2\beta^2}
-\frac{1}{2\beta\rho_2(L)},~
\epsilon_4=2\beta^2+\frac{1}{2},\\
\epsilon_5&=\frac{1}{4}-\frac{1}{2\beta}(\frac{1}{\beta}+\frac{1}{\rho_2(L)}+\frac{\alpha}{\beta})L_f^2,\\
\epsilon_6&=\frac{1}{\beta^2}(1+\frac{1}{\rho_2(L)}+\frac{\alpha}{\beta})L_f^2+\frac{L_f(1+L_f)}{2},\\
\epsilon_7&=\eta\min\{\epsilon_1-\eta\epsilon_2,~\epsilon_3-\eta\epsilon_4,
~\epsilon_5-\eta\epsilon_6,~\frac{1}{4}\},\\
\epsilon_8&=\frac{\alpha+\beta}{2\beta}+\frac{1}{2\rho_2(L)}.
\end{align*}
\end{theorem}
\begin{proof}
The proof is given in Appendix~\ref{nonconvex:proof-thm-sm}.
\end{proof}
\begin{remark}\label{nonconvex:remark:sm}
We should point out that the settings on the parameters  $\alpha$, $\beta$, and $\eta$ are just sufficient conditions. With some modifications of the proofs, other forms of settings for these algorithm parameters still can guarantee the same kind of convergence rate. Moreover, the interval $(\beta+\kappa_1,\kappa_2\beta]$ is nonempty due to the settings that $\beta>\kappa_1/(\kappa_2-1)$ and $\kappa_2>1$.
From \eqref{nonconvex:thm-sm-equ1}, we know that $\min_{k\in[T]}\{\|\bsx_{k}-\bar{\bsx}_k\|^2
+\|\bar{\bsg}_{k}^0\|^2\}=\mathcal{O}(1/T)$. In other words, Algorithm~\ref{nonconvex:algorithm-pdgd} finds a stationary point of the nonconvex optimization problem \eqref{nonconvex:eqn:xopt} with a rate $\mathcal{O}(1/T)$.
This rate is the same as that achieved by  the distributed gradient tracking algorithm \cite{daneshmand2018second} and the xFILTER algorithm  \cite{sun2019distributed} under the same assumptions on the cost functions. However, as discussed in Remark~\ref{nonconvex:remark:alg},  at each iteration, the distributed gradient tracking algorithm requires double amount of communication and the xFILTER algorithm requires more communication as well as more computation.
From \eqref{nonconvex:thm-sm-equ2}, we know that the cost difference between the global optimum and the resulting stationary point is bounded.
\end{remark}


With Assumption~\ref{nonconvex:ass:fil}, the following result states that Algorithm~\ref{nonconvex:algorithm-pdgd} can find a global optimum and the convergence rate is linear.
\begin{theorem}\label{nonconvex:thm-ft}
Suppose that Assumptions~\ref{nonconvex:ass:graph}--\ref{nonconvex:ass:fil} hold. Let $\{\bsx_k\}$ be the sequence generated by Algorithm~\ref{nonconvex:algorithm-pdgd} with  the same $\alpha$, $\beta$, and $\eta$ given in Theorem~\ref{nonconvex:thm-sm}. Then,
\begin{align}\label{nonconvex:thm-ft-equ1}
\|\bsx_{k}-\bar{\bsx}_k\|^2+n(f(\bar{x}_k)-f^*)
\le(1-\epsilon)^{k}c,~\forall k\in\mathbb{N}_0,
\end{align}
where
\begin{align*}
\epsilon&=\frac{\epsilon_{10}}{\epsilon_8},~
c=\frac{\epsilon_8\hat{V}_0}{\epsilon_9},~
\epsilon_9=\min\{\frac{1}{2\rho(L)},~\frac{\alpha-\beta}{2\alpha}\},\\
\epsilon_{10}&=\eta\min\{\epsilon_1-\eta\epsilon_2,~\epsilon_3-\eta\epsilon_4,~\frac{\nu}{2}\}.
\end{align*}
Moreover, if the projection operator $\calP_{\mathbb{X}^*}(\cdot)$ is well defined, then \begin{align}\label{nonconvex:thm-ft-equ2}
\|\bsx_{k}-{\bf 1}_n\otimes\calP_{\mathbb{X}^*}(\bar{x}_k)\|^2
\le(1-\epsilon)^{k}c(1+\frac{1}{2\nu}),~\forall k\in\mathbb{N}_0.
\end{align}
\end{theorem}
\begin{proof}
The proof is given in Appendix~\ref{nonconvex:proof-thm-ft}.
\end{proof}

\begin{remark}
From Theorems~\ref{nonconvex:thm-sm} and \ref{nonconvex:thm-ft}, we know that the considered first-order primal--dual algorithm uses the same algorithm parameters for the cases without and with the P--{\L} condition.
The proofs of both theorems are based on the same appropriately designed Lyapunov function given in Lemma~\ref{noncovex:lemma:pdgd} in the appendix. In the literature that considered distributed nonconvex optimization, e.g.,  \cite{hong2017prox,sun2018distributed,sun2019distributed,hajinezhad2019perturbed,pmlr-v80-hong18a}, the lower bounded potential functions (which may be negative) are commonly used to analyze the convergence properties of the proposed algorithms. Therefore, the analysis in those studies cannot be extended to show linear convergence when the P--{\L} condition holds since the lower bounded potential functions may not be Lyapunov functions. In the literature that obtained linear convergence for distributed optimization, e.g., \cite{lu2012zero,shi2014linear,kia2015distributed,ling2015dlm,shi2015extra,jakovetic2015linear,mokhtari2016dqm,
makhdoumi2017convergence,Damiano-TAC2016,nedic2017achieving,mansoori2019flexible,
zeng2017extrapush,xi2017dextra,qu2018harnessing,
qu2017accelerated,xi2018add,xu2018convergence,Yi2018distributed,
tian2018asy,maros2019q,jakovetic2019unification,
berahas2018balancing,liang2019exponential,yi2019exponential}, the convexity and/or the uniqueness of the global minimizer are the key in the analysis. Therefore, the analysis in those studies cannot be extended to show linear convergence when strong convexity is relaxed by the P--{\L} condition since the later does not imply convexity of cost functions and the uniqueness of the global minimizers.
\end{remark}

\begin{remark}
The distributed first-order algorithms proposed in \cite{lu2012zero,kia2015distributed,shi2014linear,ling2015dlm,shi2015extra,jakovetic2015linear,mokhtari2016dqm,
makhdoumi2017convergence,Damiano-TAC2016,nedic2017achieving,mansoori2019flexible,
zeng2017extrapush,xi2017dextra,qu2018harnessing,
qu2017accelerated,xi2018add,xu2018convergence,Yi2018distributed,
tian2018asy,yi2019exponential,maros2019q,jakovetic2019unification,berahas2018balancing,liang2019exponential} also established linear convergence.
However, in \cite{lu2012zero,kia2015distributed,shi2014linear,ling2015dlm,jakovetic2015linear,mokhtari2016dqm,
makhdoumi2017convergence,nedic2017achieving,
qu2018harnessing,qu2017accelerated,xi2018add,xu2018convergence,mansoori2019flexible,
jakovetic2019unification,berahas2018balancing}, it was assumed that each local cost function is strongly convex. In \cite{Damiano-TAC2016,maros2019q}, it was assumed that each local cost function is convex and the global cost function is strongly convex. In \cite{tian2018asy}, it was assumed that the global cost function is strongly convex. In \cite{shi2015extra,Yi2018distributed}, it was assumed that each local cost function is convex, the global cost function is restricted strongly convex, and $\mathbb{X}^*$ is a singleton. In \cite{zeng2017extrapush,xi2017dextra}, it was assumed that each local cost function is restricted strongly convex and the optimal set $\mathbb{X}^*$ is a singleton.  In \cite{liang2019exponential}, it was assumed that each local cost function is convex and  the primal--dual gradient map is metrically subregular. In \cite{yi2019exponential}, it was assumed that the global cost function satisfies the restricted secant inequality condition and the gradients of each local cost function at optimal points are the same. In contrast, the linear convergence result established in Theorem~\ref{nonconvex:thm-ft} only requires that the global cost function satisfies the P--{\L} condition, but the convexity assumption on cost functions and the singleton assumption on the optimal set and the set of each local cost function's gradients at the optimal points are not required. Moreover, it should be highlighted that the P--{\L} constant $\nu$ is not needed when implementing Algorithm~\ref{nonconvex:algorithm-pdgd}.  This is an important property since it is normally difficult to determine the P--{\L} constant. Compared with  some of the aforementioned studies, one potential drawback is that we assume the communication graph is static and undirected.
We leave the extension to time-varying directed graph for future work.  
\end{remark}

\begin{remark}\label{nonconvex:remark3}
If we use  $\bsL\bsx={\bf 0}_{np}$ as the constraint in \eqref{nonconvex:eqn:xoptcon}, then
we can construct an alternative distributed primal--dual gradient descent algorithm
\begin{subequations}\label{nonconvex:elia-algo-dc}
\begin{align}
x_{i,k+1} &= x_{i,k}-\eta(\sum_{j=1}^nL_{ij}(\alpha x_{j,k}+\beta v_{j,k})+\nabla f_i(x_{i,k})), \label{nonconvex:elia-algo-x-dc}\\
v_{i,k+1} &=v_{i,k}+ \eta\beta\sum_{j=1}^n L_{ij}x_{j,k},\forall x_{i,0},~v_{i,0}\in\mathbb{R}^p. \label{nonconvex:elia-algo-q-dc}
\end{align}
\end{subequations}
Similar results as shown in Theorems~\ref{nonconvex:thm-sm} and \ref{nonconvex:thm-ft} (as well as the results stated in Theorems~\ref{zero:thm-determin-sm} and \ref{zero:thm-determin-ft} in the next section) can be obtained. We omit the details due to the space limitation.
Different from the requirement that $\sum_{j=1}^nv_{j,0}={\bf 0}_p$ in the algorithm \eqref{nonconvex:kia-algo-dc}, $v_{i,0}$ can be arbitrarily chosen in the algorithm  \eqref{nonconvex:elia-algo-dc}. In other words, the algorithm  \eqref{nonconvex:elia-algo-dc} is robust to the initial condition $v_{i,0}$. However, it requires additional communication of $v_{j,k}$ in \eqref{nonconvex:elia-algo-x-dc}, compared to \eqref{nonconvex:kia-algo-dc}.
\end{remark}

\section{Distributed deterministic zeroth-order primal--dual algorithm}\label{nonconvex:sec-main-zo}
In this section, we propose a distributed deterministic zeroth-order primal--dual algorithm and analyze its convergence rate.

\subsection{Algorithm Description}
When implementing the first-order algorithm \eqref{nonconvex:kia-algo-dc}, each agent needs to know the gradient of its local cost function. However, in some practical applications, the explicit expressions of the gradients are unavailable or difficult to obtain \cite{conn2009introduction}.
Inspired by the deterministic gradient estimator \eqref{zero:determin-gradient}, based on  our proposed  distributed first-order algorithm \eqref{nonconvex:kia-algo-dc}, we propose the following zeroth-order algorithm:
\begin{subequations}\label{zero:determin-dc}
\begin{align}
x_{i,k+1} &= x_{i,k}-\eta(\alpha\sum_{j=1}^nL_{ij}x_{j,k}+\beta v_{i,k}+\hat{\nabla}f_i(x_{i,k},\delta_{i,k})), \label{zero:determin-dc-x}\\
v_{i,k+1} &=v_{i,k}+ \eta\beta\sum_{j=1}^n L_{ij}x_{j,k},\forall x_{i,0}\in\mathbb{R}^p, ~
\sum_{j=1}^nv_{j,0}={\bf 0}_p,  \label{zero:determin-dc-q}
\end{align}
\end{subequations}
where $\hat{\nabla}f_i(x_{i,k},\delta_{i,k})$ is the deterministic estimator of $\nabla f_i(x_{i,k})$ as defined in \eqref{zero:determin-gradient}.
Note that the gradient estimator  $\hat{\nabla}f_i(x_{i,k},\delta_{i,k})$ can be calculated by sampling the function values of $f_i$ at $p+1$ points.

We present  the distributed deterministic zeroth-order primal--dual algorithm \eqref{zero:determin-dc} in pseudo-code as Algorithm~\ref{zero:algorithm-determin}.
\begin{algorithm}[tb]
\caption{Distributed Deterministic Zeroth-Order Primal--Dual Algorithm}
\label{zero:algorithm-determin}
\begin{algorithmic}[1]
\STATE \textbf{Input}: parameters $\alpha>0$, $\beta>0$, $\eta>0$, and $\{\delta_{i,k}>0\}$.
\STATE \textbf{Initialize}: $ x_{i,0}\in\mathbb{R}^p$ and $v_{i,0}={\bf 0}_p,~
\forall i\in[n]$.
\FOR{$k=0,1,\dots$}
\FOR{$i=1,\dots,n$  in parallel}
\STATE  Broadcast $x_{i,k}$ to $\mathcal{N}_i$ and receive $x_{j,k}$ from $j\in\mathcal{N}_i$;
\STATE Sample $f_i(x_{i,k})$ and $\{f_i(x_{i,k}+\delta_{i,k}{\bf e}_l)\}_{l=1}^p$;
\STATE  Update $x_{i,k+1}$ by \eqref{zero:determin-dc-x};
\STATE  Update $v_{i,k+1}$ by \eqref{zero:determin-dc-q}.
\ENDFOR
\ENDFOR
\STATE  \textbf{Output}: $\{\bsx_{k}\}$.
\end{algorithmic}
\end{algorithm}

\begin{remark}\label{zero:remark:determin}
A different distributed deterministic zeroth-order algorithm was proposed in \cite{tang2019distributed}. However, in that algorithm, at each iteration each agent $i$ needs to communicate two additional $p$-dimensional variables besides the communication of $x_{i,k}$ with its neighbors, which results in a heavy communication burden when $p$ is large. Moreover, the deterministic gradient estimator used in \cite{tang2019distributed} requires that at each iteration each agent samples its local cost function values at $2p$ points compared with  $p+1$ points used in our algorithm.
\end{remark}

\subsection{Convergence Analysis}
In this section, we provide convergence analysis for both  without and with Assumption~\ref{nonconvex:ass:fil}.

We use the same notations introduced in Section~\ref{nonconvex:sec-main-dc}. Moreover, denote $h_{i,k}=\hat{\nabla}f_i(x_{i,k},\delta_{i,k})$,  $\bsh_k=\col(h_{1,k},\dots,h_{n,k})$, $\bar{\bsh}_k=\bsH\bsh_k$, $\delta_{k}=\max_{i\in[n]}\{\delta_{i,k}\}$, $h_{i,k}^0=\hat{\nabla}f_i(\bar{x}_{k},\delta_{k})$, $\bsh^0_k=\col(h_{1,k}^0,\dots,h_{n,k}^0)$, $\bar{\bsh}_k^0=\bsH\bsh_k^0$, and
\begin{align*}
\hat{U}_k=\|\bm{x}_k\|^2_{\bsK}+\|\bsv_k
+\frac{1}{\beta}\bsh_k^0\|^2_{\bsK}+n(f(\bar{x}_k)-f^*).
\end{align*}

We have the following convergence results for Algorithm~\ref{zero:algorithm-determin} without Assumption~\ref{nonconvex:ass:fil}.
\begin{theorem}\label{zero:thm-determin-sm}
Suppose that Assumptions~\ref{nonconvex:ass:graph}--\ref{nonconvex:ass:fiu} hold. Let $\{\bsx_k\}$ be the sequence generated by Algorithm~\ref{zero:algorithm-determin} with $\alpha\in(\beta+\tilde{\kappa}_1,\kappa_2\beta]$, $\beta>\max\{\frac{\tilde{\kappa}_1}{\kappa_2-1},~\kappa_3,~\tilde{\kappa}_{4}\}$, $\eta\in(0,\min\{\frac{\tilde{\epsilon}_{1}}{\tilde{\epsilon}_{2}},~\frac{\epsilon_{3}}{\epsilon_{4}},
~\frac{\tilde{\epsilon}_{5}}{\tilde{\epsilon}_{6}}\})$, and $\delta_{i,k}>0$ such that
\begin{align}\label{zero:determin-delta}
\delta^a_i=\sum_{k=0}^{+\infty}\delta^2_{i,k}<+\infty,
\end{align}
then
\begin{align}
&\sum_{k=0}^{T}(\|\bsx_{k}-\bar{\bsx}_k\|^2
+\|\bar{\bsg}_{k}^0\|^2)
\le \frac{\tilde{c}}{\tilde{\epsilon}_{7}},~\forall T\in\mathbb{N}_0,\label{zero:thm-determin-sm-equ1}\\
&f(\bar{x}_{T+1})-f^*
\le \frac{\tilde{c}}{n},~\forall T\in\mathbb{N}_0,\label{zero:thm-determin-sm-equ2}
\end{align}
where
\begin{align*}
\tilde{\kappa}_1&=\frac{1}{2\rho_2(L)}(2+9L_f^2),\\
\tilde{\kappa}_{4}&=6(\kappa_2+\frac{1}{\rho_2(L)})L_f^2
+2(9(\kappa_2+\frac{1}{\rho_2(L)})^2L_f^4+3L_f^2)^{\frac{1}{2}},\\
\tilde{\epsilon}_{1}&=(\alpha-\beta)\rho_2(L)-\frac{1}{2}(2+9L_f^2),\\
\tilde{\epsilon}_2&=\beta^2\rho(L)+(2\alpha^2+\beta^2)\rho^2(L)+\frac{15}{2}L_f^2,\\
\tilde{\epsilon}_{5}&=\frac{1}{8}-\frac{3}{2\beta}(\frac{1}{\beta}+\frac{1}{\rho_2(L)}+\frac{\alpha}{\beta})L_f^2,\\
\tilde{\epsilon}_6&=\frac{3}{\beta^2}(1+\frac{1}{\rho_2(L)}+\frac{\alpha}{\beta})L_f^2
+\frac{L_f(1+3L_f)}{2},\\
\tilde{\epsilon}_{7}&=\eta\min\{\tilde{\epsilon}_{1}-\eta\tilde{\epsilon}_{2},
~\frac{1}{8}\},~\epsilon_{11}=(\frac{15\eta}{4}+5\eta^2)
\frac{3npL_f^2}{4}+\epsilon_{12},\\
\epsilon_{12}&=((\frac{1}{\beta^2}+\frac{1}{2\eta\beta})
(\frac{1}{\rho_2(L)}+\frac{\alpha}{\beta})
+\frac{1}{2\eta\beta^2}+\frac{1}{\beta^2}+\frac{1}{2})
\frac{3npL_f^2}{4},\\
\tilde{c}&=\epsilon_8\hat{U}_0+(\epsilon_{11}+\epsilon_{12})\sum_{i=1}^{n}\delta^a_i.
\end{align*}
\end{theorem}
\begin{proof}
The proof is given in Appendix~\ref{zero:proof-thm-determin-sm}.
\end{proof}
\begin{remark}
Similar to the discussion in Remark~\ref{nonconvex:remark:sm},
from \eqref{zero:thm-determin-sm-equ1}, we know that Algorithm~\ref{zero:algorithm-determin} finds a stationary point of the nonconvex optimization problem \eqref{nonconvex:eqn:xopt} with a rate $\mathcal{O}(1/T)$.
This rate is the same as that achieved by the distributed stochastic zeroth-order algorithm proposed in \cite{hajinezhad2019zone} under different assumptions.
More specifically, \cite{hajinezhad2019zone} considers a more realistic scenario where the cost function values are queried with noises. However, in \cite{hajinezhad2019zone}, it needs an additional assumption that the gradient of each local cost function is bounded and each agent needs to employ $\mathcal{O}(T)$ function value samplings at each iteration.
From \eqref{zero:thm-determin-sm-equ2}, we know that the cost difference between the global optimum and the resulting stationary point is bounded.
\end{remark}

With Assumption~\ref{nonconvex:ass:fil}, the following result states that Algorithm~\ref{zero:algorithm-determin} can find a global optimum and the convergence rate is linear.
\begin{theorem}\label{zero:thm-determin-ft}
Suppose that Assumptions~\ref{nonconvex:ass:graph}--\ref{nonconvex:ass:fil} hold. Let $\{\bsx_k\}$ be the sequence generated by Algorithm~\ref{zero:algorithm-determin} with the same $\alpha$, $\beta$, and $\eta$ given in Theorem~\ref{zero:thm-determin-sm},  and $\delta_{i,k}\in(0,\hat{\epsilon}^{\frac{k}{2}}]$, then
\begin{align}\label{zero:thm-determin-ft-equ1}
&\|\bsx_{k}-\bar{\bsx}_k\|^2+n(f(\bar{x}_k)-f^*)
\le\frac{1}{\epsilon_9}((1-\tilde{\epsilon})^{k+1}\epsilon_8\hat{U}_0
+\phi(\tilde{\epsilon},\hat{\epsilon},\breve{\epsilon})),~\forall k\in\mathbb{N}_0,
\end{align}
 where $\hat{\epsilon}\in(0,1)$, $\breve{\epsilon}\in(\hat{\epsilon},1)$,
\begin{align*}
\tilde{\epsilon}&=\frac{\tilde{\epsilon}_{10}}{\epsilon_{8}},~
\tilde{\epsilon}_{10}=\eta\min\{\tilde{\epsilon}_{1}-\eta\tilde{\epsilon}_{2},
~\epsilon_{3}-\eta\epsilon_{4},~\frac{\nu}{4}\},\\
\phi(\tilde{\epsilon},\hat{\epsilon},\breve{\epsilon})&=(\frac{\epsilon_{11}}{1-\tilde{\epsilon}}+\epsilon_{12})
\begin{cases}
  \frac{(1-\tilde{\epsilon})^{k+1}}{1-\tilde{\epsilon}-\hat{\epsilon}}, & \mbox{if } 1-\tilde{\epsilon}>\hat{\epsilon} \\
  \frac{\hat{\epsilon}^{k+1}}{\hat{\epsilon}+\tilde{\epsilon}-1}, & \mbox{if } 1-\tilde{\epsilon}<\hat{\epsilon} \\
  \frac{\breve{\epsilon}^{k+1}}{\breve{\epsilon}-\hat{\epsilon}}, & \mbox{if } 1-\tilde{\epsilon}=\hat{\epsilon}.
\end{cases}
\end{align*}
Moreover, if the projection operator $\calP_{\mathbb{X}^*}(\cdot)$ is well defined, then \begin{align}\label{zero:thm-determin-ft-equ2}
\|\bsx_{k}-{\bf 1}_n\otimes\calP_{\mathbb{X}^*}(\bar{x}_k)\|^2
\le\frac{1}{\epsilon_9}(1+\frac{1}{2\nu})((1-\epsilon)^{k+1}\epsilon_8\hat{U}_0
+\phi(\tilde{\epsilon},\hat{\epsilon},\breve{\epsilon})),~\forall k\in\mathbb{N}_0.
\end{align}
\end{theorem}
\begin{proof}
The proof is given in Appendix~\ref{zero:proof-thm-determin-ft}.
\end{proof}

\begin{remark}
It is straightforward to see that $\phi=\mathcal{O}(a^k)$, where $a=\max\{1-\tilde{\epsilon},\hat{\epsilon},\breve{\epsilon}\}<1$, so $\phi$ linearly converges to zero.
By comparing Theorems~\ref{nonconvex:thm-sm} and \ref{nonconvex:thm-ft} with Theorems~\ref{zero:thm-determin-sm} and \ref{zero:thm-determin-ft}, respectively, we see that the proposed distributed first- and zeroth-order algorithms have the same convergence properties under the same assumptions. Similar convergence results  as stated in Theorems~\ref{zero:thm-determin-sm} and \ref{zero:thm-determin-ft} were also achieved by
the distributed deterministic zeroth-order algorithm proposed in \cite{tang2019distributed} under the same assumptions. Compared with \cite{tang2019distributed}, in addition to the advantages discussed in Remark~\ref{zero:remark:determin}, one more potential advantage of Theorem~\ref{zero:thm-determin-ft} is that the P--{\L} constant $\nu$ is not used. However, \cite{tang2019distributed} also proposed a distributed random zeroth-order algorithm. We expect that the proposed distributed first-order algorithm \eqref{nonconvex:kia-algo-dc} can be extended to be a random zeroth-order algorithm with two noisy samples of local cost function values by each agent at each iteration. Such an extension is our ongoing work.
\end{remark}

\section{Simulations}\label{nonconvex:sec-simulation}
In this section, we verify and illustrate the theoretical results through numerical simulations.

We consider the nonconvex distributed binary classification problem in \cite{sun2019distributed}, which is formulated as the optimization problem \eqref{nonconvex:eqn:xopt} with each component function $f_i$ given by
\begin{align*}
f_i(x)&=\frac{1}{m}\sum_{s=1}^{m}\log(1+\exp(-y_{is}x^\top z_{is}))
+\sum_{l=1}^{p}\frac{\lambda\mu([x]_l)^2}{1+\mu([x]_l)^2},
\end{align*}
where $m$ is the number of data points of each sensor, $y_{is}\in\{1,-1\}$ denotes the label for the $s$-th data point of sensor $i$, $z_{is}\in\mathbb{R}^p$ is the feature vector, and $\lambda$ as well as $\mu$ are regularization parameters. All settings for cost functions and the communication graph are the same as those described in \cite{sun2019distributed}. Specifically, $n=20$, $p=50$, $m=200$, $\lambda=0.001$, and $\mu=1$. The graph used in the simulation is the random geometric graph and the graph parameter is set to be $0.5$. We independently and randomly generate $nm$ data points with dimension $p$ and each agent contains $m$ data points.

We compare Algorithms~\ref{nonconvex:algorithm-pdgd} and \ref{zero:algorithm-determin} with state-of-the-art algorithms: distributed gradient descent (DGD) with diminishing stepsizes \cite{zeng2018nonconvex,daneshmand2018second}, distributed first-order gradient tracking algorithm (DFO-GTA) \cite{qu2018harnessing,daneshmand2018second}, distributed deterministic zeroth-order gradient tracking algorithm (DDZO-GTA) \cite{tang2019distributed}, xFILTER \cite{sun2019distributed},  Prox-GPDA \cite{hong2017prox}, and D-GPDA \cite{sun2018distributed}.

We use $$P(T)=\min_{k\in [T]}\{\|\nabla f(\bar{x}_k)\|^ {2}+\frac{1}{n}\sum_{i=1}^{n}\|x_{i,k}-\bar{x}_k\|^ {2}\}$$ to measure the performance of each algorithm.
Fig.~\ref{nonconvex:fig:dbc} illustrates the convergence of $P(T)$ with respect to the number of communication rounds $T$ for these algorithms with the same initial condition. It can be seen that our first-order algorithm (Algorithm~\ref{nonconvex:algorithm-pdgd}) gives the best performance in general. We also see that both zeroth-order algorithms (Algorithm~\ref{zero:algorithm-determin} and DDZO-GTA \cite{tang2019distributed}) exhibit almost identical behavior as their first-order counterparts (Algorithm~\ref{nonconvex:algorithm-pdgd} and  DFO-GTA \cite{qu2018harnessing,daneshmand2018second}) during the early stage, but then slow down and converge at a sublinear rate.

\begin{figure}[!ht]
\centering
  \includegraphics[width=0.8\textwidth]{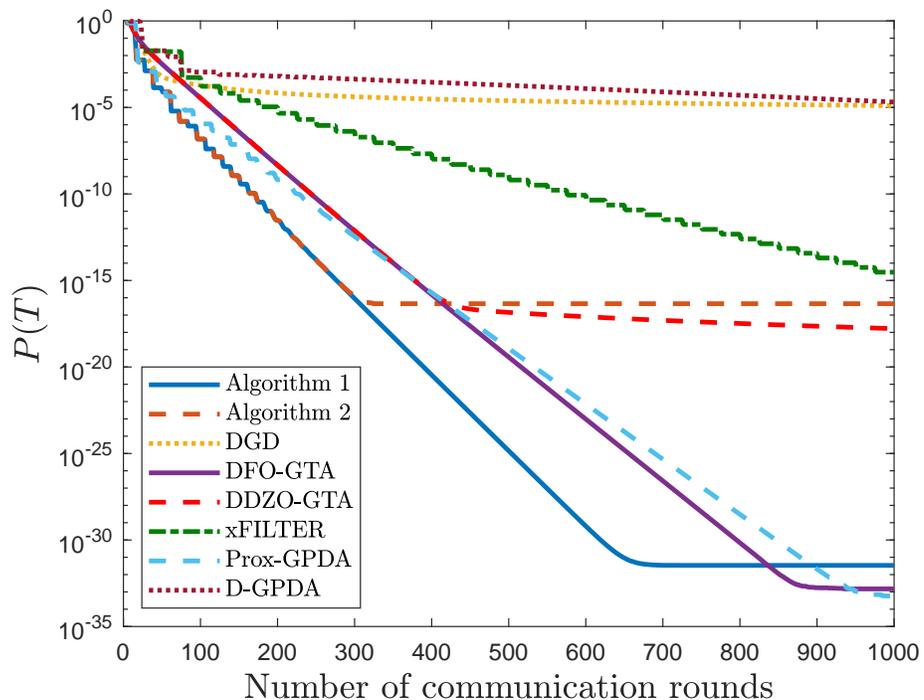}
  \caption{Evolutions of $P(T)$ with respect to the number of communication rounds $T$.}
  \label{nonconvex:fig:dbc}
\end{figure}

In order to compare the performance of the two deterministic zeroth-order algorithms (Algorithm~\ref{zero:algorithm-determin} and DDZO-GTA \cite{tang2019distributed}), we plot the convergence of $P(T)$ with respect to the number of function value queries and variables communicated in Fig.~\ref{nonconvex:fig:dbc_zero} and Fig.~\ref{nonconvex:fig:dbc_zero2}, respectively. It can be seen that Algorithm~\ref{zero:algorithm-determin} gives better performance.

\begin{figure}[!ht]
\centering
  \includegraphics[width=0.8\textwidth]{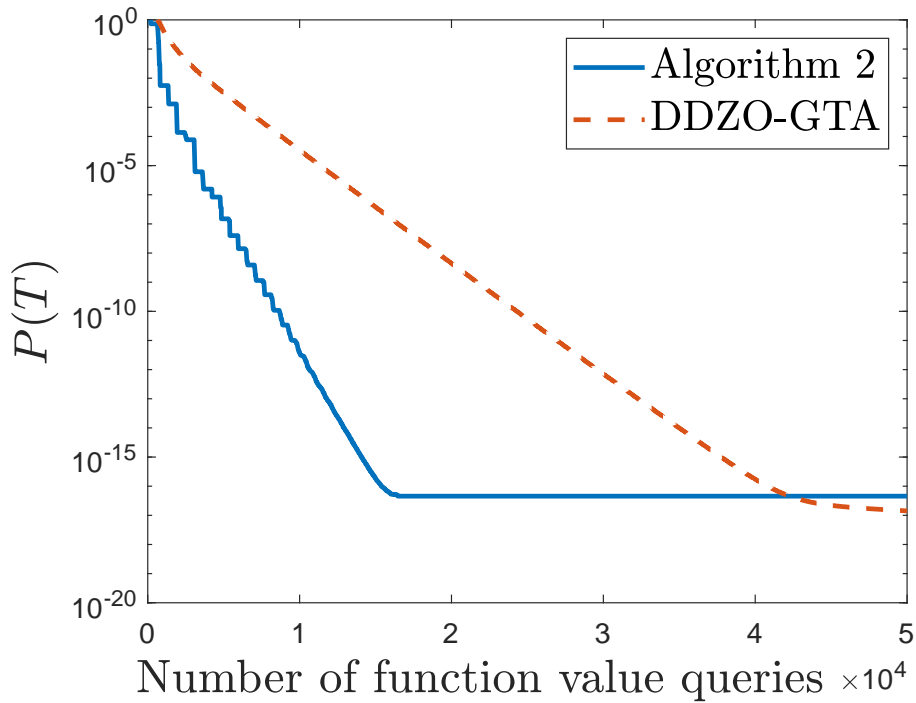}
  \caption{Evolutions of $P(T)$ with respect to the number of function value queries $N$.}
  \label{nonconvex:fig:dbc_zero}
\end{figure}

\begin{figure}[!ht]
\centering
  \includegraphics[width=0.8\textwidth]{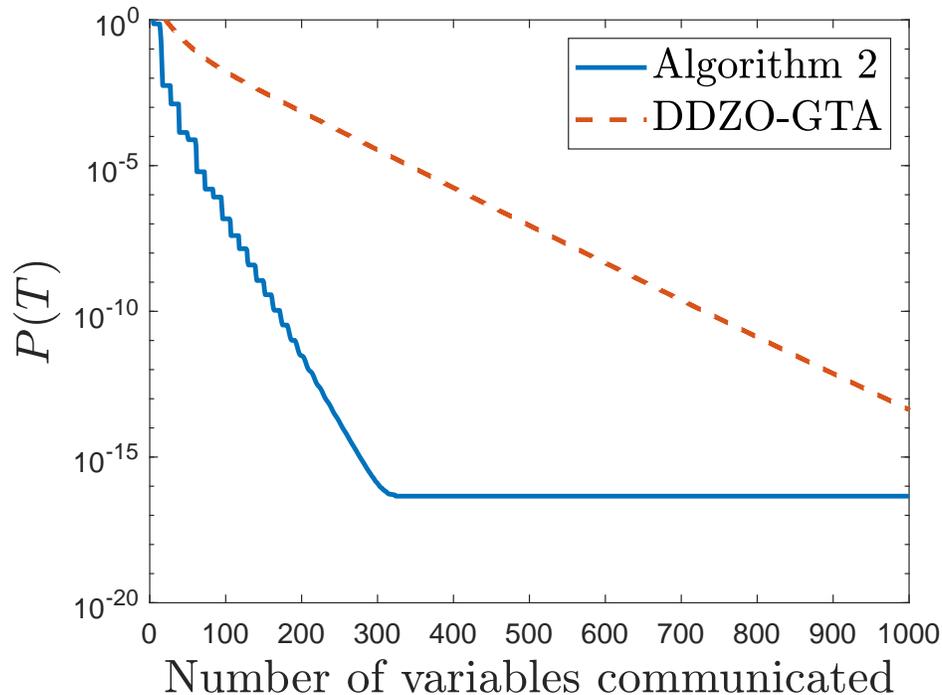}
  \caption{Evolutions of $P(T)$ with respect to the number of variables communicated.}
  \label{nonconvex:fig:dbc_zero2}
\end{figure}

\section{Conclusions}\label{nonconvex:sec-conclusion}
In this paper, we studied distributed nonconvex optimization. We proposed distributed first- and zeroth-order primal--dual algorithms and derived their convergence rates. Linear convergence was established when the global cost function  satisfies the P--{\L} condition. This relaxes the standard strong convexity condition in the literature.  Interesting directions for future work include proving linear convergence rate for larger stepsizes, considering time-varying graphs, investigating the scenarios where the function values are sampled with noises, and studying constraints.

\section*{Acknowledgments}
The authors would like to thank Drs. Mingyi Hong, Na Li, Haoran Sun, and Yujie Tang, for sharing their codes.

\bibliographystyle{IEEEtran}
\bibliography{refextra}









\appendix

\subsection{Useful Lemmas}\label{zero:app-lemmas}

The following results are used in the proofs.

\begin{lemma}\label{nonconvex:lemma-Xinlei} (Lemmas~1 and 2 in \cite{yi2018distributedarXiv})
Let $L$ be the Laplacian matrix of the connected graph $\mathcal{G}$ and $K_n={\bf I}_n-\frac{1}{n}{\bf 1}_n{\bf 1}^{\top}_n$.
Then $L$ and $K_n$ are positive semi-definite, $\nullrank(L)=\nullrank(K_n)=\{{\bf 1}_n\}$, $L\le\rho(L){\bf I}_n$, $\rho(K_n)=1$,
\begin{align}
&K_nL=LK_n=L,\label{nonconvex:KL-L-eq}\\
&0\le\rho_2(L)K_n\le L\le\rho(L)K_n.\label{nonconvex:KL-L-eq2}
\end{align}
Moreover, there exists an orthogonal matrix $[r \ R]\in \mathbb{R}^{n \times n}$ with $r=\frac{1}{\sqrt{n}}\mathbf{1}_n$ and $R \in \mathbb{R}^{n\times (n-1)}$ such that
\begin{align}
&R\Lambda_1^{-1}R^{\top}L=LR\Lambda_1^{-1}R^{\top}=K_n,\label{nonconvex:lemma-eq}\\
&\frac{1}{\rho(L)}K_n\leq R\Lambda_1^{-1}R^{\top}\le\frac{1}{\rho_2(L)}K_n,\label{nonconvex:lemma-eq2}
\end{align}
where $\Lambda_1=\diag([\lambda_2,\dots,\lambda_n])$ with $0<\lambda_2\leq\dots\leq\lambda_n$ being the  nonzero eigenvalues of the Laplacian matrix $L$.
\end{lemma}

\begin{lemma}\label{zerosg:lemma:sumgeo}
Let $a,b\in(0,1)$ be two constants, then
\begin{align}\label{zerosg:lemma:sumgeo-equ}
\sum_{\tau=0}^{k}a^\tau b^{k-\tau}\le
\begin{cases}
  \frac{a^{k+1}}{a-b}, & \mbox{if } a>b \\
  \frac{b^{k+1}}{b-a}, & \mbox{if } a<b \\
  \frac{c^{k+1}}{c-b}, & \mbox{if } a=b ,
\end{cases}
\end{align}
where $c$ is any constant in $(a,1)$.
\end{lemma}
\begin{proof}
If $a>b$, then
\begin{align*}
\sum_{\tau=0}^{k}a^\tau b^{k-\tau}=a^k\sum_{\tau=0}^{k}(\frac{b}{a})^{k-\tau}
\le \frac{a^{k+1}}{a-b}.
\end{align*}
Similarly, when $a<b$, we have
\begin{align*}
\sum_{\tau=0}^{k}a^\tau b^{k-\tau}=b^k\sum_{\tau=0}^{k}(\frac{a}{b})^\tau
\le \frac{b^{k+1}}{b-a}.
\end{align*}
If $a=b$, then for any $c\in(a,1)$, we have
\begin{align*}
\sum_{\tau=0}^{k}a^\tau b^{k-\tau}\le\sum_{\tau=0}^{k}c^\tau b^{k-\tau}=c^k\sum_{\tau=0}^{k}(\frac{b}{c})^{k-\tau}
\le \frac{c^{k+1}}{c-b}.
\end{align*}
Hence, this lemma holds.
\end{proof}

\subsection{Proof of Theorem~\ref{nonconvex:thm-sm}}\label{nonconvex:proof-thm-sm}
To prove Theorem~\ref{nonconvex:thm-sm}, the following lemma is used, which presents the general relations of two consecutive outputs of Algorithm~\ref{nonconvex:algorithm-pdgd}.
\begin{lemma}\label{noncovex:lemma:pdgd}
Suppose Assumptions~\ref{nonconvex:ass:graph}--\ref{nonconvex:ass:fiu} hold. Let $\{\bsx_k\}$ be the sequence generated by Algorithm~\ref{nonconvex:algorithm-pdgd} with $\alpha>\beta$. Then,
\begin{align}\label{nonconvex:vkLya-sm2}
V_{k+1}
&\le  V_{k}-\|\bsx_k\|^2_{\eta(\epsilon_1-\eta\epsilon_2)\bsK}
-\|\bm{v}_k+\frac{1}{\beta}\bsg_k^0\|^2_{\eta(\epsilon_3-\eta\epsilon_4)\bsK}-\eta(\epsilon_5-\eta\epsilon_6)\|\bar{\bsg}_{k}\|^2
-\frac{\eta}{4}\|\bar{\bsg}_{k}^0\|^2,
\end{align}
where
\begin{align*}
V_{k}&=\frac{1}{2}\|\bm{x}_k \|^2_{\bsK}+\frac{1}{2}\|\bsv_k
+\frac{1}{\beta}\bsg_k^0\|^2_{\bsQ+\frac{\alpha}{\beta}\bsK}+\bsx_k^\top\bsK(\bm{v}_k+\frac{1}{\beta}\bsg_k^0)
+n(f(\bar{x}_k)-f^*),
\end{align*}
and $\bsQ=R\Lambda^{-1}_1R^{\top}\otimes {\bf I}_p$ with matrices $R$ and $\Lambda^{-1}_1$  given in Lemma~\ref{nonconvex:lemma-Xinlei}.
\end{lemma}
\begin{proof}
We first note that $V_{k}$ is well defined since $f^*>-\infty$ as assumed in Assumption~\ref{nonconvex:ass:optset}.

Denote $\bar{v}_k=\frac{1}{n}({\bf 1}_n^\top\otimes{\bf I}_p)\bsv_k$. Then,
from \eqref{nonconvex:kia-algo-dc-q}, we know that
$\bar{v}_{k+1}=\bar{v}_k$.
This together with the fact that $\sum_{i=1}^{n}v_{i,0}={\bf0}_p$ implies
$\bar{v}_k={\bf 0}_p$.
Then, from $\bar{v}_k={\bf 0}_p$ and \eqref{nonconvex:kia-algo-dc-x}, we know that
\begin{align}
&\bar{\bsx}_{k+1}=\bar{\bsx}_{k}-\eta\bar{\bsg}_k.\label{nonconvex:xbardynamic}
\end{align}

Noting that $\nabla\tilde{f}$ is Lipschitz-continuous with constant $L_{f}>0$ as assumed in Assumption~\ref{nonconvex:ass:fiu}, we have
\begin{align}
\|\bsg^0_{k}-\bsg_{k}\|^2\le L_f^2\|\bar{\bsx}_{k}-\bsx_{k}\|^2=L_f^2\|\bsx_{k}\|^2_{\bsK}.\label{nonconvex:gg1}
\end{align}
Then, from \eqref{nonconvex:gg1} and $\rho(\bsH)=1$, we have
\begin{align}
&\|\bar{\bsg}^0_k-\bar{\bsg}_k\|^2=\|\bsH(\bsg^0_{k}-\bsg_{k})\|^2
\le\|\bsg^0_{k}-\bsg_{k}\|^2\le L_f^2\|\bsx_k\|^2_{\bsK}.\label{nonconvex:gg2}
\end{align}

From $\nabla\tilde{f}$ is Lipschitz-continuous and \eqref{nonconvex:xbardynamic}, we have
\begin{align}
&\|\bsg^0_{k+1}-\bsg^0_{k}\|^2\le L_f^2\|\bar{\bsx}_{k+1}-\bar{\bsx}_{k}\|^2=\eta^2L_f^2\|\bar{\bsg}_k\|^2.\label{nonconvex:gg}
\end{align}

We have
\begin{align}
\frac{1}{2}\|\bm{x}_{k+1} \|^2_{\bsK}
&=\frac{1}{2}\|\bm{x}_k-\eta(\alpha\bsL\bm{x}_k+\beta\bm{v}_k+\bsg_k) \|^2_{\bsK}\nonumber\\
&=\frac{1}{2}\|\bm{x}_k\|^2_{\bsK}-\eta\alpha\|\bsx_k\|^2_{\bsL}
+\frac{\eta^2\alpha^2}{2}\|\bsx_k\|^2_{\bsL^2}
\nonumber\\
&\quad-\eta\beta\bsx^\top_k({\bf I}_{np}-\eta\alpha\bsL)\bsK(\bm{v}_k+\frac{1}{\beta}\bsg_k)
+\frac{\eta^2\beta^2}{2}\|\bm{v}_k+\frac{1}{\beta}\bsg_k\|^2_{\bsK}\nonumber\\
&=\frac{1}{2}\|\bm{x}_k\|^2_{\bsK}-\|\bsx_k\|^2_{\eta\alpha\bsL
-\frac{\eta^2\alpha^2}{2}\bsL^2}\nonumber\\
&\quad-\eta\beta\bsx^\top_k({\bf I}_{np}-\eta\alpha\bsL)\bsK(\bm{v}_k
+\frac{1}{\beta}\bsg_k^0
+\frac{1}{\beta}\bsg_k-\frac{1}{\beta}\bsg_k^0)\nonumber\\
&\quad+\frac{\eta^2\beta^2}{2}\|\bm{v}_k+\frac{1}{\beta}\bsg_k^0
+\frac{1}{\beta}\bsg_k-\frac{1}{\beta}\bsg_k^0\|^2_{\bsK}\nonumber\\
&\le\frac{1}{2}\|\bm{x}_k\|^2_{\bsK}-\|\bsx_k\|^2_{\eta\alpha\bsL
-\frac{\eta^2\alpha^2}{2}\bsL^2}-\eta\beta\bsx^\top_k\bsK(\bm{v}_k+\frac{1}{\beta}\bsg_k^0)\nonumber\\
&\quad
+\frac{\eta}{2}\|\bm{x}_k\|^2_{\bsK}
+\frac{\eta}{2}\|\bsg_k-\bsg_k^0\|^2
+\frac{\eta^2\alpha^2}{2}\|\bm{x}_k\|^2_{\bsL^2}
+\frac{\eta^2\beta^2}{2}\|\bm{v}_k+\frac{1}{\beta}\bsg_k^0\|^2_{\bsK}\nonumber\\
&\quad+\frac{\eta^2\alpha^2}{2}\|\bm{x}_k\|^2_{\bsL^2}
+\frac{\eta^2}{2}\|\bsg_k-\bsg_k^0\|^2
+\eta^2\beta^2\|\bm{v}_k+\frac{1}{\beta}\bsg_k^0\|^2_{\bsK}
+\eta^2\|\bsg_k-\bsg_k^0\|^2\nonumber\\
&=\frac{1}{2}\|\bm{x}_k\|^2_{\bsK}-\|\bsx_k\|^2_{\eta\alpha\bsL-\frac{\eta}{2}\bsK
-\frac{3\eta^2\alpha^2}{2}\bsL^2}
+\frac{\eta}{2}(1+3\eta)\|\bsg_k-\bsg_k^0\|^2
\nonumber\\
&\quad-\eta\beta\bsx^\top_k\bsK(\bm{v}_k+\frac{1}{\beta}\bsg_k^0)
+\|\bm{v}_k+\frac{1}{\beta}\bsg_k^0\|^2_{\frac{3\eta^2\beta^2}{2}\bsK}\label{nonconvex:v1k-zero}\\
&\le\frac{1}{2}\|\bm{x}_k\|^2_{\bsK}-\|\bsx_k\|^2_{\eta\alpha\bsL-\frac{\eta}{2}\bsK
-\frac{3\eta^2\alpha^2}{2}\bsL^2-\frac{\eta}{2}(1+3\eta)L_f^2\bsK}
\nonumber\\
&\quad-\eta\beta\bsx^\top_k\bsK(\bm{v}_k+\frac{1}{\beta}\bsg_k^0)
+\|\bm{v}_k+\frac{1}{\beta}\bsg_k^0\|^2_{\frac{3\eta^2\beta^2}{2}\bsK},\label{nonconvex:v1k}
\end{align}
where the first equality holds since \eqref{nonconvex:kia-algo-dc-compact-x}; the second equality holds since \eqref{nonconvex:KL-L-eq}; the first inequality holds since the Cauchy--Schwarz inequality and $\rho(\bsK)=1$; and the last  inequality holds since \eqref{nonconvex:gg1}.

We have
\begin{align}
\frac{1}{2}\|\bsv_{k+1}+\frac{1}{\beta}\bsg_{k+1}^0\|^2_{\bsQ+\frac{\alpha}{\beta}\bsK}
&=\frac{1}{2}\|\bm{v}_k+\frac{1}{\beta}\bsg_{k}^0+\eta\beta\bsL\bm{x}_k
+\frac{1}{\beta}(\bsg_{k+1}^0-\bsg_{k}^0) \|^2_{\bsQ+\frac{\alpha}{\beta}\bsK}\nonumber\\
&=\frac{1}{2}\|\bm{v}_k+\frac{1}{\beta}\bsg_{k}^0\|^2_{\bsQ+\frac{\alpha}{\beta}\bsK}
+\eta\bsx^\top_k(\beta\bsK+\alpha\bsL)(\bm{v}_k+\frac{1}{\beta}\bsg_k^0)\nonumber\\
&~~~+\|\bsx_k\|^2_{\frac{\eta^2\beta}{2}(\beta\bsL+\alpha\bsL^2)}
+\frac{1}{2\beta^2}\|\bsg_{k+1}^0-\bsg_{k}^0\|^2_{\bsQ+\frac{\alpha}{\beta}\bsK}\nonumber\\
&~~~+\frac{1}{\beta}(\bm{v}_k+\frac{1}{\beta}\bsg_{k}^0
+\eta\beta\bsL\bm{x}_k)^\top(\bsQ+\frac{\alpha}{\beta}\bsK)(\bsg_{k+1}^0-\bsg_{k}^0)\nonumber\\
&\le\frac{1}{2}\|\bm{v}_k+\frac{1}{\beta}\bsg_{k}^0\|^2_{\bsQ+\frac{\alpha}{\beta}\bsK}
+\eta\bsx^\top_k(\beta\bsK+\alpha\bsL)(\bm{v}_k+\frac{1}{\beta}\bsg_k^0)\nonumber\\
&~~~+\|\bsx_k\|^2_{\frac{\eta^2\beta}{2}(\beta\bsL+\alpha\bsL^2)}
+\frac{1}{2\beta^2}\|\bsg_{k+1}^0-\bsg_{k}^0\|^2_{\bsQ+\frac{\alpha}{\beta}\bsK}\nonumber\\
&~~~+\frac{\eta}{2\beta}\|\bm{v}_k+\frac{1}{\beta}\bsg_{k}^0\|^2_{\bsQ+\frac{\alpha}{\beta}\bsK}
+\frac{1}{2\eta\beta}\|\bsg_{k+1}^0-\bsg_{k}^0\|^2_{\bsQ+\frac{\alpha}{\beta}\bsK}\nonumber\\
&~~~+\frac{\eta^2\beta^2}{2}\|\bsL\bm{x}_k\|^2_{\bsQ+\frac{\alpha}{\beta}\bsK}
+\frac{1}{2\beta^2}\|\bsg_{k+1}^0-\bsg_{k}^0\|^2_{\bsQ+\frac{\alpha}{\beta}\bsK}\nonumber\\
&=\frac{1}{2}\|\bm{v}_k+\frac{1}{\beta}\bsg_{k}^0\|^2_{\bsQ+\frac{\alpha}{\beta}\bsK}
+\eta\bsx^\top_k(\beta\bsK+\alpha\bsL)(\bm{v}_k+\frac{1}{\beta}\bsg_k^0)\nonumber\\
&~~~+\|\bsx_k\|^2_{\eta^2\beta(\beta\bsL+\alpha\bsL^2)}
+\|\bm{v}_k+\frac{1}{\beta}\bsg_{k}^0
\|^2_{\frac{\eta}{2\beta}(\bsQ+\frac{\alpha}{\beta}\bsK)}\nonumber\\
&~~~+(\frac{1}{\beta^2}+\frac{1}{2\eta\beta})
\|\bsg_{k+1}^0-\bsg_{k}^0\|^2_{\bsQ+\frac{\alpha}{\beta}\bsK}\nonumber\\
&\le\frac{1}{2}\|\bm{v}_k+\frac{1}{\beta}\bsg_{k}^0\|^2_{\bsQ+\frac{\alpha}{\beta}\bsK}
+\eta\bsx^\top_k(\beta\bsK+\alpha\bsL)(\bm{v}_k+\frac{1}{\beta}\bsg_k^0)\nonumber\\
&~~~+\|\bsx_k\|^2_{\eta^2\beta(\beta\bsL+\alpha\bsL^2)}
+\|\bm{v}_k+\frac{1}{\beta}\bsg_{k}^0
\|^2_{\frac{\eta}{2\beta}(\bsQ+\frac{\alpha}{\beta}\bsK)}\nonumber\\
&~~~+c_1
\|\bsg_{k+1}^0-\bsg_{k}^0\|^2\label{nonconvex:v2k-zero}\\
&\le\frac{1}{2}\|\bm{v}_k+\frac{1}{\beta}\bsg_{k}^0\|^2_{\bsQ+\frac{\alpha}{\beta}\bsK}
+\eta\bsx^\top_k(\beta\bsK+\alpha\bsL)(\bm{v}_k+\frac{1}{\beta}\bsg_k^0)\nonumber\\
&~~~+\|\bsx_k\|^2_{\eta^2\beta(\beta\bsL+\alpha\bsL^2)}
+\|\bm{v}_k+\frac{1}{\beta}\bsg_{k}^0
\|^2_{\frac{\eta}{2\beta}(\bsQ+\frac{\alpha}{\beta}\bsK)}\nonumber\\
&~~~+\eta^2 c_1L_f^2\|\bar{\bsg}_k\|^2,\label{nonconvex:v2k}
\end{align}
where $c_1=(\frac{1}{\beta^2}+\frac{1}{2\eta\beta})
(\frac{1}{\rho_2(L)}+\frac{\alpha}{\beta})$; the first equality holds since \eqref{nonconvex:kia-algo-dc-compact-v}; the second equality holds since \eqref{nonconvex:KL-L-eq} and \eqref{nonconvex:lemma-eq};  the first inequality holds since the Cauchy--Schwarz inequality; the last equality holds since \eqref{nonconvex:KL-L-eq} and \eqref{nonconvex:lemma-eq}; the second inequality holds since $\rho(\bsQ+\frac{\alpha}{\beta}\bsK)\le\rho(\bsQ)+\frac{\alpha}{\beta}\rho(\bsK)$, \eqref{nonconvex:lemma-eq2}, $\rho(\bsK)=1$; and the last  inequality holds since \eqref{nonconvex:gg}.

We have
\begin{align}
\bsx_{k+1}^\top\bsK(\bm{v}_{k+1}+\frac{1}{\beta}\bsg_{k+1}^0)
&=(\bm{x}_k-\eta(\alpha\bsL\bm{x}_k+\beta\bm{v}_k+\bsg_k^0+\bsg_k-\bsg_k^0))^\top
\bsK(\bm{v}_k\nonumber\\
&~~~+\frac{1}{\beta}\bsg_{k}^0+\eta\beta\bsL\bm{x}_k+\frac{1}{\beta}(\bsg_{k+1}^0-\bsg_{k}^0))\nonumber\\
&=\bm{x}_k^\top(\bsK-\eta(\alpha+\eta\beta^2)\bsL)(\bm{v}_k+\frac{1}{\beta}\bsg_{k}^0)
\nonumber\\
&~~~+\|\bm{x}_k\|^2_{\eta\beta(\bsL-\eta\alpha\bsL^2)}
+\frac{1}{\beta}\bm{x}_k^\top(\bsK-\eta\alpha\bsL)(\bsg_{k+1}^0-\bsg_{k}^0)\nonumber\\
&~~~-\eta\beta\|\bm{v}_k+\frac{1}{\beta}\bsg_{k}^0\|^2_{\bsK}
-\eta(\bm{v}_k+\frac{1}{\beta}\bsg_{k}^0)^\top\bsK(\bsg_{k+1}^0-\bsg_{k}^0)\nonumber\\
&~~~-\eta(\bsg_k-\bsg_k^0)^\top
\bsK(\bm{v}_k+\frac{1}{\beta}\bsg_{k}^0+\eta\beta\bsL\bm{x}_k\nonumber\\
&~~~+\frac{1}{\beta}(\bsg_{k+1}^0-\bsg_{k}^0))\nonumber\\
&\le\bm{x}_k^\top(\bsK-\eta\alpha\bsL)(\bm{v}_k+\frac{1}{\beta}\bsg_{k}^0)
+\frac{\eta^2\beta^2}{2}\|\bsL\bsx_k\|^2\nonumber\\
&~~~+\frac{\eta^2\beta^2}{2}\|\bm{v}_k+\frac{1}{\beta}\bsg_{k}^0\|^2_{\bsK}
+\|\bm{x}_k\|^2_{\eta\beta(\bsL-\eta\alpha\bsL^2)}\nonumber\\
&~~~+\frac{\eta}{2}\|\bm{x}_k\|^2_\bsK+\frac{1}{2\eta\beta^2}\|\bsg_{k+1}^0-\bsg_{k}^0\|^2
+\frac{\eta^2\alpha^2}{2}\|\bsL\bm{x}_k\|^2\nonumber\\
&~~~+\frac{1}{2\beta^2}\|\bsg_{k+1}^0-\bsg_{k}^0\|^2
-\eta\beta\|\bm{v}_k+\frac{1}{\beta}\bsg_{k}^0\|^2_{\bsK}\nonumber\\
&~~~+\frac{\eta^2}{2}\|\bm{v}_k+\frac{1}{\beta}\bsg_{k}^0\|^2_{\bsK}
+\frac{1}{2}\|\bsg_{k+1}^0-\bsg_{k}^0\|^2\nonumber\\
&~~~+\frac{\eta}{2}\|\bsg_k-\bsg_k^0\|^2
+\frac{\eta}{2}\|\bm{v}_k+\frac{1}{\beta}\bsg_{k}^0\|^2_{\bsK}
+\frac{\eta^2}{2}\|\bsg_k-\bsg_k^0\|^2\nonumber\\
&~~~+\frac{\eta^2\beta^2}{2}\|\bsL\bm{x}_k\|^2
+\frac{\eta^2}{2}\|\bsg_k-\bsg_k^0\|^2
+\frac{1}{2\beta^2}\|\bsg_{k+1}^0-\bsg_{k}^0\|^2\nonumber\\
&=\bm{x}_k^\top(\bsK-\eta\alpha\bsL)(\bm{v}_k+\frac{1}{\beta}\bsg_{k}^0)
+\frac{\eta}{2}(1+2\eta)\|\bsg_k-\bsg_k^0\|^2\nonumber\\
&~~~+\|\bm{x}_k\|^2_{\eta(\beta\bsL+\frac{1}{2}\bsK)
+\eta^2(\frac{\alpha^2}{2}-\alpha\beta+\beta^2)\bsL^2}\nonumber\\
&~~~+c_2\|\bsg_{k+1}^0-\bsg_{k}^0\|^2
-\|\bm{v}_k+\frac{1}{\beta}\bsg_{k}^0\|^2_{\eta(\beta-\frac{1}{2}
-\frac{\eta}{2}-\frac{\eta\beta^2}{2})\bsK}\label{nonconvex:v3k-zero}\\
&\le\bm{x}_k^\top\bsK(\bm{v}_k+\frac{1}{\beta}\bsg_{k}^0)
-\eta\alpha\bm{x}_k^\top\bsL(\bm{v}_k+\frac{1}{\beta}\bsg_{k}^0)\nonumber\\
&~~~+\|\bm{x}_k\|^2_{\eta(\beta\bsL+\frac{1}{2}\bsK)
+\eta^2(\frac{\alpha^2}{2}-\alpha\beta+\beta^2)\bsL^2
+\frac{\eta}{2}(1+2\eta)L_f^2\bsK}\nonumber\\
&~~~+\eta^2c_2L_f^2\|\bar{\bsg}_{k}\|^2
-\|\bm{v}_k+\frac{1}{\beta}\bsg_{k}^0\|^2_{\eta(\beta-\frac{1}{2}-\frac{\eta}{2}
-\frac{\eta\beta^2}{2})\bsK}
,\label{nonconvex:v3k}
\end{align}
where $c_2=\frac{1}{2\eta\beta^2}+\frac{1}{\beta^2}+\frac{1}{2}$; the first equality holds since \eqref{nonconvex:kia-algo-dc-compact}; the second equality holds since \eqref{nonconvex:KL-L-eq}; the first inequality holds since the Cauchy--Schwarz inequality, \eqref{nonconvex:KL-L-eq}, and $\rho(\bsK)=1$; and the last  inequality holds since \eqref{nonconvex:gg} and \eqref{nonconvex:gg1}.

We have
\begin{align}
n(f(\bar{x}_{k+1})-f^*)&=\tilde{f}(\bar{\bsx}_{k+1})-nf^*
=\tilde{f}(\bar{\bsx}_k)-nf^*+\tilde{f}(\bar{\bsx}_{k+1})-\tilde{f}(\bar{\bsx}_k)\nonumber\\
&\le\tilde{f}(\bar{\bsx}_k)-nf^*
-\eta\bar{\bsg}_{k}^\top\bsg^0_k
+\frac{\eta^2L_f}{2}\|\bar{\bsg}_{k}\|^2\nonumber\\
&=\tilde{f}(\bar{\bsx}_k)-nf^*
-\eta\bar{\bsg}_{k}^\top\bar{\bsg}^0_k
+\frac{\eta^2L_f}{2}\|\bar{\bsg}_{k}\|^2\nonumber\\
&=n(f(\bar{x}_k)-f^*)
-\frac{\eta}{2}\bar{\bsg}_{k}^\top(\bar{\bsg}_k+\bar{\bsg}^0_k-\bar{\bsg}_k)\nonumber\\
&~~~-\frac{\eta}{2}(\bar{\bsg}_{k}-\bar{\bsg}^0_k+\bar{\bsg}^0_k)^\top\bar{\bsg}^0_k
+\frac{\eta^2L_f}{2}\|\bar{\bsg}_{k}\|^2\nonumber\\
&\le n(f(\bar{x}_k)-f^*)-\frac{\eta}{4}\|\bar{\bsg}_{k}\|^2
+\frac{\eta}{4}\|\bar{\bsg}^0_k-\bar{\bsg}_k\|^2
-\frac{\eta}{4}\|\bar{\bsg}_{k}^0\|^2\nonumber\\
&~~~+\frac{\eta}{4}\|\bar{\bsg}^0_k-\bar{\bsg}_k\|^2
+\frac{\eta^2L_f}{2}\|\bar{\bsg}_{k}\|^2\nonumber\\
&=n(f(\bar{x}_k)-f^*)-\frac{\eta}{4}(1-2\eta L_f)\|\bar{\bsg}_{k}\|^2
+\frac{\eta}{2}\|\bar{\bsg}^0_k-\bar{\bsg}_k\|^2
-\frac{\eta}{4}\|\bar{\bsg}_{k}^0\|^2\label{nonconvex:v4k-zero}\\
&\le n(f(\bar{x}_k)-f^*)-\frac{\eta}{4}(1-2\eta L_f)\|\bar{\bsg}_{k}\|^2
+\|\bsx_k\|^2_{\frac{\eta}{2}L_f^2\bsK}
-\frac{\eta}{4}\|\bar{\bsg}_{k}^0\|^2,\label{nonconvex:v4k-1}
\end{align}
where the first inequality holds since that $\tilde{f}$ is smooth, \eqref{nonconvex:lemma:smooth} and \eqref{nonconvex:xbardynamic}; the third equality holds since $\bar{\bsg}_{k}^\top\bsg^0_k=\bsg_{k}^\top\bsH\bsg^0_k=\bsg_{k}^\top\bsH\bsH\bsg^0_k=\bar{\bsg}_{k}^\top\bar{\bsg}^0_k$; the second inequality holds since the Cauchy--Schwarz inequality; and the last inequality holds since \eqref{nonconvex:gg2}.

From \eqref{nonconvex:v1k}, \eqref{nonconvex:v2k}, \eqref{nonconvex:v3k}, and \eqref{nonconvex:v4k-1}, we have
\begin{align}
V_{k+1}
&\le  V_{k}-\|\bsx_k\|^2_{\eta\bsM_1-\eta^2\bsM_2}
-\|\bm{v}_k+\frac{1}{\beta}\bsg_k^0\|^2_{\eta\bsM_3-\eta^2\bsM_4}
-\eta(\epsilon_5-\eta\epsilon_6)\|\bar{\bsg}_{k}\|^2
-\frac{\eta}{4}\|\bar{\bsg}_{k}^0\|^2,\label{nonconvex:vkLya}
\end{align}
where
\begin{align*}
\bsM_1&=(\alpha-\beta)\bsL-\frac{1}{2}(2+3L_f^2)\bsK,\\
\bsM_2&=\beta^2\bsL+(2\alpha^2+\beta^2)\bsL^2+\frac{5}{2}L_f^2\bsK,\\
\bsM_3&=(\beta-\frac{1}{2}-\frac{\alpha}{2\beta^2})\bsK-\frac{1}{2\beta}\bsQ,\\
\bsM_4&=(2\beta^2+\frac{1}{2})\bsK.
\end{align*}

From $\alpha>\beta$, \eqref{nonconvex:KL-L-eq2}, and \eqref{nonconvex:lemma-eq2},  we have
\begin{align}
\bsM_1&=(\alpha-\beta)\bsL-\frac{1}{2}(2+3L_f^2)\bsK\nonumber\\
&\ge(\alpha-\beta)\rho_2(L)\bsK-\frac{1}{2}(2+3L_f^2)\bsK=\epsilon_1\bsK,\label{nonconvex:m1}\\
\bsM_2&=\beta^2\bsL+(2\alpha^2+\beta^2)\bsL^2+\frac{5}{2}L_f^2\bsK
\le\epsilon_2\bsK,\label{nonconvex:m2}\\
\bsM_3&=(\beta-\frac{1}{2}-\frac{\alpha}{2\beta^2})\bsK-\frac{1}{2\beta}\bsQ\nonumber\\
&\ge(\beta-\frac{1}{2}-\frac{\alpha}{2\beta^2})\bsK-\frac{1}{2\beta\rho_2(L)}\bsK
=\epsilon_3\bsK.\label{nonconvex:m3}
\end{align}

From \eqref{nonconvex:vkLya} and \eqref{nonconvex:m1}--\eqref{nonconvex:m3}, we know that \eqref{nonconvex:vkLya-sm2} holds.
\end{proof}

We are now ready to prove Theorem~\ref{nonconvex:thm-sm}.

Denote $\epsilon_9=\min\{\frac{1}{2\rho(L)},~\frac{\alpha-\beta}{2\alpha}\}$.
We know
\begin{align}
V_{k}&\ge\frac{1}{2}\|\bsx_{k}\|^2_{\bsK}
+\frac{1}{2}(\frac{1}{\rho(L)}+\frac{\alpha}{\beta})\|\bsv_k+\frac{1}{\beta}\bsg_k^0\|^2_{\bsK}\nonumber\\
&\quad-\frac{\beta}{2\alpha}\|\bsx_{k}\|^2_{\bsK}
-\frac{\alpha}{2\beta}\|\bsv_k+\frac{1}{\beta}\bsg_k^0\|^2_{\bsK}
+n(f(\bar{x}_k)-f^*)\nonumber\\
&\ge\epsilon_{9}(\|\bsx_{k}\|^2_{\bsK}+\|\bsv_k+\frac{1}{\beta}\bsg_k^0\|^2_{\bsK})
+n(f(\bar{x}_k)-f^*)\label{nonconvex:vkLya3.2}\\
&\ge\epsilon_{9}\hat{V}_k\ge0,\label{nonconvex:vkLya3}
\end{align}
where the first inequality holds since \eqref{nonconvex:lemma-eq2} and the Cauchy--Schwarz inequality; and the last inequality holds since $0<\epsilon_{9}<1$. Similarly, we have
\begin{align}\label{nonconvex:vkLya3.1}
V_k\le\epsilon_{8}\hat{V}_k.
\end{align}

From $\beta+\kappa_1<\alpha$ and $\kappa_1=\frac{1}{2\rho_2(L)}(2+3L_f^2)$, we have
\begin{align}\label{nonconvex:beta1}
\epsilon_1>0.
\end{align}
From $\alpha\le\kappa_2\beta$ and $\beta>\kappa_3$, we have
\begin{align}\label{nonconvex:beta2}
\epsilon_3\ge(\beta-\frac{1}{2}-\frac{\kappa_2}{2\beta})-\frac{1}{2\beta\rho_2(L)}>0.
\end{align}
From $\alpha\le\kappa_2\beta$ and $\beta>\kappa_4$, we have
\begin{align}\label{nonconvex:beta3}
\epsilon_5
\ge\frac{1}{4}-\frac{1}{2\beta}(\frac{1}{\beta}
+\frac{1}{\rho_2(L)}+\kappa_2)L_f^2
>0.
\end{align}

From \eqref{nonconvex:beta1}--\eqref{nonconvex:beta3},  and $0<\eta<\min\{\frac{\epsilon_1}{\epsilon_2},~\frac{\epsilon_3}{\epsilon_4},~\frac{\epsilon_5}{\epsilon_6}\}$, we have
\begin{align}
&\eta(\epsilon_1-\eta\epsilon_2)>0,\label{nonconvex:vkLya1.1}\\
&\eta(\epsilon_3-\eta\epsilon_4)>0,\label{nonconvex:vkLya1.2}\\
&\eta(\epsilon_5-\eta\epsilon_6)>0.\label{nonconvex:vkLya1}
\end{align}
Then, from \eqref{nonconvex:vkLya1.1}--\eqref{nonconvex:vkLya1}, we have
\begin{align}\label{nonconvex:vkLya1.3}
\epsilon_7>0.
\end{align}

From \eqref{nonconvex:vkLya-sm2}, we have
\begin{align}\label{nonconvex:vkLya-sm4}
\sum_{k=0}^{T}V_{k+1}
&\le  \sum_{k=0}^{T}V_{k}-\sum_{k=0}^{T}\|\bsx_k\|^2_{\eta(\epsilon_1-\eta\epsilon_2)\bsK}
-\sum_{k=0}^{T}\|\bm{v}_k+\frac{1}{\beta}\bsg_k^0\|^2_{\eta(\epsilon_3-\eta\epsilon_4)\bsK}\nonumber\\
&\quad-\sum_{k=0}^{T}\eta(\epsilon_5-\eta\epsilon_6)\|\bar{\bsg}_{k}\|^2
-\sum_{k=0}^{T}\frac{\eta}{4}\|\bar{\bsg}_{k}^0\|^2.
\end{align}

Hence, from \eqref{nonconvex:vkLya-sm4} and \eqref{nonconvex:vkLya3}, we have
\begin{align}\label{nonconvex:vkLya-sm5}
V_{T+1}+\epsilon_7\sum_{k=0}^{T}W_k\le  V_{0}
\le \epsilon_8\hat{V}_{0}.
\end{align}

From \eqref{nonconvex:vkLya-sm5}, \eqref{nonconvex:vkLya1.3}, and \eqref{nonconvex:vkLya3}, we know that \eqref{nonconvex:thm-sm-equ1} holds.

From \eqref{nonconvex:vkLya-sm5}, \eqref{nonconvex:vkLya1.3}, and \eqref{nonconvex:vkLya3.2}, we know that  \eqref{nonconvex:thm-sm-equ2} holds.

\subsection{Proof of Theorem~\ref{nonconvex:thm-ft}}\label{nonconvex:proof-thm-ft}
In this proof, we use the same notations as those used in the proof of Theorem~\ref{nonconvex:thm-sm}.

From \eqref{nonconvex:vkLya3}, we have
\begin{align}\label{nonconvex:vkLya5}
\|\bsx_{k}-\bar{\bsx}_k\|^2+n(f(\bar{x}_k)-f^*)
\le\hat{V}_k\le\frac{V_{k}}{\epsilon_9}.
\end{align}

From Assumptions~\ref{nonconvex:ass:optset} and \ref{nonconvex:ass:fil} as well as \eqref{nonconvex:equ:plc}, we have that
\begin{align}\label{nonconvex:gg3}
\|\bar{\bsg}^0_k\|^2=n\|\nabla f(\bar{x}_k)\|^2\ge2\nu n(f(\bar{x}_k)-f^*).
\end{align}

(i) 
From \eqref{nonconvex:vkLya1.1} and \eqref{nonconvex:vkLya1.2}, we have
\begin{align}\label{nonconvex:vkLya1.4}
\epsilon_{10}>0~\text{and}~\epsilon=\frac{\epsilon_{10}}{\epsilon_{8}}>0.
\end{align}

Noting that $\epsilon_3<\beta$, $\epsilon_4>2\beta^2$, and $\epsilon_8>\frac{\alpha+\beta}{2\beta}>1$, we have
\begin{align}\label{nonconvex:vkLya-epsilon}
0<\epsilon=\frac{\epsilon_{10}}{\epsilon_{8}}
\le\frac{\eta(\epsilon_3-\eta\epsilon_4)}{\epsilon_8}
\le\frac{\epsilon_3^2}{4\epsilon_4\epsilon_8}<\frac{1}{8}.
\end{align}

Then, from \eqref{nonconvex:vkLya-sm2},  \eqref{nonconvex:vkLya1}, and \eqref{nonconvex:gg3}, we have
\begin{align}\label{nonconvex:vkLya2}
V_{k+1}\le V_{k}-\hat{V}_k\eta\min\{\epsilon_1-\eta\epsilon_2,
~\epsilon_3-\eta\epsilon_4,~\frac{\nu}{2}\}.
\end{align}
From \eqref{nonconvex:vkLya2}, \eqref{nonconvex:vkLya1.4}, and \eqref{nonconvex:vkLya3.1}, we have
\begin{align}\label{nonconvex:vkLya2.1}
V_{k+1}\le V_{k}-\epsilon_{10}\hat{V}_k\le V_{k}-\frac{\epsilon_{10}}{\epsilon_{8}}V_{k}.
\end{align}

From \eqref{nonconvex:vkLya2.1} and  \eqref{nonconvex:vkLya-epsilon}, we have
\begin{align}\label{nonconvex:vkLya4}
V_{k+1}&\le  (1-\epsilon)V_k
\le (1-\epsilon)^{k+1}V_0\le(1-\epsilon)^{k+1}\epsilon_8\hat{V}_0.
\end{align}

Hence, \eqref{nonconvex:vkLya4} and \eqref{nonconvex:vkLya5}
give \eqref{nonconvex:thm-ft-equ1}.

(ii) If the projection operator $\calP_{\mathbb{X}^*}(\cdot)$ is well defined, then from the Cauchy--Schwarz inequality and \eqref{nonconvex:lemma:plc-equ}, we know that
\begin{align}\label{nonconvex:vkLya6}
\|\bsx_{k}-{\bf 1}_n\otimes\calP_{\mathbb{X}^*}(\bar{x}_k)\|^2
&=\|\bsx_{k}-\bar{\bsx}_k+\bar{\bsx}_k-{\bf 1}_n\otimes\calP_{\mathbb{X}^*}(\bar{x}_k)\|^2\nonumber\\
&\le(1+\frac{1}{2\nu})\|\bsx_{k}-\bar{\bsx}_k\|^2
+(1+2\nu)n\|\bar{x}_k-\calP_{\mathbb{X}^*}(\bar{x}_k)\|^2\nonumber\\
&\le(1+\frac{1}{2\nu})\|\bsx_{k}-\bar{\bsx}_k\|^2
+(1+2\nu)\frac{n}{2\nu}(f(\bar{x}_k)-f^*)\nonumber\\
&=(1+\frac{1}{2\nu})(\|\bsx_{k}\|^2_{\bsK}+n(f(\bar{x}_k)-f^*)).
\end{align}

Finally, \eqref{nonconvex:thm-ft-equ1} and \eqref{nonconvex:vkLya6} yield \eqref{nonconvex:thm-ft-equ2}.

\subsection{Proof of Theorem~\ref{zero:thm-determin-sm}}\label{zero:proof-thm-determin-sm}
The proof is similar to the proof of Theorem~\ref{nonconvex:thm-sm} with some modifications.
Lemma~\ref{noncovex:lemma:pdgd} is replaced by the following lemma\red{.}

\begin{lemma}\label{zero:lemma-determine-sm}
Let $\{\bsx_k\}$ be the sequence generated by Algorithm~\ref{zero:algorithm-determin}. If Assumptions~\ref{nonconvex:ass:graph}--\ref{nonconvex:ass:fiu} hold with $\alpha>\beta$. Then,
\begin{align}\label{zero:ukLya2}
U_{k+1}
&\le U_{k}-\|\bsx_k\|^2_{\eta(\tilde{\epsilon}_1-\eta\tilde{\epsilon}_2)\bsK}
-\|\bm{v}_k+\frac{1}{\beta}\bsh_k^0\|^2_{\eta(\epsilon_3-\eta\epsilon_{4})\bsK}\nonumber\\
&\quad-\eta(\tilde{\epsilon}_{5}-\eta\tilde{\epsilon}_{6})\|\bar{\bsh}_{k}\|^2
-\frac{\eta}{8}\|\bar{\bsg}^0_k\|^2
+\epsilon_{11}\delta^2_{k}+\epsilon_{12}\delta^2_{k+1},
\end{align}
where
\begin{align*}
U_{k}&=\frac{1}{2}\|\bm{x}_k \|^2_{\bsK}
+\frac{1}{2}\|\bsv_k+\frac{1}{\beta}\bsh_k^0\|^2_{\bsQ+\frac{\alpha}{\beta}\bsK}
+\bsx_k^\top\bsK(\bm{v}_k+\frac{1}{\beta}\bsh_k^0)+n(f(\bar{x}_k)-f^*).
\end{align*}
\end{lemma}
\begin{proof}
The distributed deterministic zeroth-order algorithm \eqref{zero:determin-dc} can be rewritten as \begin{subequations}\label{zero:kia-algo-dc-compact}
\begin{align}
\bm{x}_{k+1}&=\bm{x}_k-\eta(\alpha\bsL\bm{x}_k+\beta\bm{v}_k+\bsh_k),\label{zero:kia-algo-dc-compact-x}\\
\bm{v}_{k+1}&=\bm{v}_k+\eta\beta\bsL\bm{x}_k,~\forall \bsx_0\in\mathbb{R}^{np},~\bsv_0={\bf 0}_{np}.\label{zero:kia-algo-dc-compact-v}
\end{align}
\end{subequations}

We know that \eqref{nonconvex:gg1} still holds. Similar to the way to get \eqref{nonconvex:gg}, we have
\begin{align}
\|\bsg^0_{k+1}-\bsg^0_{k}\|^2\le L_f^2\|\bar{\bsx}_{k+1}-\bar{\bsx}_{k}\|^2
=\eta^2L_f^2\|\bar{\bsh}_{k}\|^2.\label{zero:gg}
\end{align}

From \eqref{zero:gradient-close}, we have
\begin{align}
&\|\bsh_k-\bsg_k\|^2\le\frac{np}{4}L_f^2\delta_k^2,\label{zero:gradient-close1}\\
&\|\bsh_k^0-\bsg_k^0\|^2\le\frac{np}{4}L_f^2\delta_k^2.\label{zero:gradient-close2}
\end{align}

We have
\begin{align}
\|\bsh^0_{k+1}-\bsh^0_{k}\|^2
&\le3\|\bsh^0_{k+1}-\bsg^0_{k+1}\|^2+
3\|\bsg^0_{k+1}-\bsg^0_{k}\|^2+3\|\bsg^0_{k}-\bsh^0_{k}\|^2\notag\\
&\le\frac{3np}{4}L_f^2\delta^2_{k+1} +3\eta^2L_f^2\|\bar{\bsh}_k\|^2+\frac{3np}{4}L_f^2\delta^2_{k} \notag\\
&=\frac{3np}{4}L_f^2(\delta^2_{k+1}+\delta^2_{k}) +3\eta^2L_f^2\|\bar{\bsh}_k\|^2,\label{zero:gg2}
\end{align}
where the first inequality holds since the Cauchy--Schwarz inequality; and the last inequality holds since \eqref{zero:gradient-close2} and \eqref{zero:gg}.
Similarly, from the Cauchy--Schwarz inequality, \eqref{zero:gradient-close1}, \eqref{zero:gradient-close2}, and \eqref{nonconvex:gg1}, we have
\begin{align}
\|\bsh^0_{k}-\bsh_{k}\|^2
\le\frac{3np}{2}L_f^2\delta^2_{k}  +3L_f^2\|\bsx_{k}\|^2_{\bsK}.\label{zero:gg3}
\end{align}
Then, from \eqref{zero:gg3} and $\rho(\bsH)=1$, we have
\begin{align}
&\|\bar{\bsh}^0_k-\bar{\bsh}_k\|^2=\|\bsH(\bsh^0_{k}-\bsh_{k})\|^2
\le\|\bsh^0_{k}-\bsh_{k}\|^2\le \frac{3np}{2}L_f^2\delta^2_{k} +3L_f^2\|\bsx_k\|^2_{\bsK}.\label{zero:gg4}
\end{align}

We have
\begin{align}
\|\bar{\bsh}^0_k\|^2&=\|\bar{\bsh}^0_k-\bar{\bsg}^0_k+\bar{\bsg}^0_k\|^2\ge-\|\bar{\bsh}^0_k-\bar{\bsg}^0_k\|^2
+\frac{1}{2}\|\bar{\bsg}^0_k\|^2\nonumber\\
&\ge-\|\bsh^0_k-\bsg^0_k\|^2
+\frac{1}{2}\|\bar{\bsg}^0_k\|^2\ge-\frac{np}{4}L_f^2\delta^2_{k}
+\frac{1}{2}\|\bar{\bsg}^0_k\|^2,\label{zero:gg5}
\end{align}
where the first inequality holds since the Cauchy--Schwarz inequality; the second inequality holds since $\rho(\bsH)=1$; and the last inequality holds since \eqref{zero:gradient-close2}.

Similar to the way to get \eqref{nonconvex:v1k-zero}, we have
\begin{align}
\frac{1}{2}\|\bm{x}_{k+1} \|^2_{\bsK}
&\le\frac{1}{2}\|\bm{x}_k\|^2_{\bsK}-\|\bsx_k\|^2_{\eta\alpha\bsL-\frac{\eta}{2}\bsK
-\frac{3\eta^2\alpha^2}{2}\bsL^2}-\eta\beta\bsx^\top_k\bsK(\bm{v}_k+\frac{1}{\beta}\bsh_k^0)\nonumber\\
&\quad
+\|\bm{v}_k+\frac{1}{\beta}\bsh_k^0\|^2_{\frac{3\eta^2\beta^2}{2}\bsK}
+\frac{\eta}{2}(1+3\eta)\|\bsh_k-\bsh_k^0\|^2
\label{zero:u1k-1}.
\end{align}
Then, from \eqref{zero:u1k-1} and \eqref{zero:gg3}, we have
\begin{align}
\frac{1}{2}\|\bm{x}_{k+1} \|^2_{\bsK}
&\le\frac{1}{2}\|\bm{x}_k\|^2_{\bsK}-\|\bsx_k\|^2_{\eta\alpha\bsL-\frac{\eta}{2}\bsK
-\frac{3\eta^2\alpha^2}{2}\bsL^2-\frac{3\eta}{2}(1+3\eta)L_f^2\bsK}
\nonumber\\
&~~~-\eta\beta\bsx^\top_k\bsK(\bm{v}_k+\frac{1}{\beta}\bsh_k^0)
+\|\bm{v}_k+\frac{1}{\beta}\bsh_k^0\|^2_{\frac{3\eta^2\beta^2}{2}\bsK}
+\frac{3np}{4}L_f^2\delta^2_{k}\eta(1+3\eta).\label{zero:u1k}
\end{align}

Similar to the way to get \eqref{nonconvex:v2k-zero}, we have
\begin{align}
\frac{1}{2}\|\bm{v}_{k+1}+\frac{1}{\beta}\bsh_{k+1}^0\|^2_{\bsQ+\frac{\alpha}{\beta}\bsK}
&\le\frac{1}{2}\|\bm{v}_k+\frac{1}{\beta}\bsh_{k}^0\|^2_{\bsQ+\frac{\alpha}{\beta}\bsK}
+\eta\bsx^\top_k(\beta\bsK+\alpha\bsL)(\bm{v}_k+\frac{1}{\beta}\bsh_k^0)\nonumber\\
&~~~+\|\bsx_k\|^2_{\eta^2\beta(\beta\bsL+\alpha\bsL^2)}
+\|\bm{v}_k+\frac{1}{\beta}\bsh_{k}^0
\|^2_{\frac{\eta}{2\beta}(\bsQ+\frac{\alpha}{\beta}\bsK)}+c_1
\|\bsh_{k+1}^0-\bsh_{k}^0\|^2.\label{zero:u2k-1}
\end{align}
Then, from \eqref{zero:u2k-1} and \eqref{zero:gg2}, we have
\begin{align}
\frac{1}{2}\|\bm{v}_{k+1}+\frac{1}{\beta}\bsh_{k+1}^0\|^2_{\bsQ+\frac{\alpha}{\beta}\bsK}
&\le\frac{1}{2}\|\bm{v}_k+\frac{1}{\beta}\bsh_{k}^0\|^2_{\bsQ+\frac{\alpha}{\beta}\bsK}
+\eta\bsx^\top_k(\beta\bsK+\alpha\bsL)(\bm{v}_k+\frac{1}{\beta}\bsh_k^0)\nonumber\\
&~~~+\|\bsx_k\|^2_{\eta^2\beta(\beta\bsL+\alpha\bsL^2)}
+\|\bm{v}_k+\frac{1}{\beta}\bsh_{k}^0
\|^2_{\frac{\eta}{2\beta}(\bsQ+\frac{\alpha}{\beta}\bsK)}\nonumber\\
&~~~+3\eta^2c_1L_f^2\|\bar{\bsh}_k\|^2
+\frac{3np}{4}c_1L_f^2(\delta^2_{k+1}+\delta^2_{k}).\label{zero:u2k}
\end{align}

Similar to the way to get \eqref{nonconvex:v3k-zero}, we have
\begin{align}
\bsx_{k+1}^\top\bsK(\bm{v}_{k+1}+\frac{1}{\beta}\bsh_{k+1}^0)
&\le\bm{x}_k^\top(\bsK-\eta\alpha\bsL)(\bm{v}_k+\frac{1}{\beta}\bsh_{k}^0)
+\frac{\eta}{2}(1+2\eta)\|\bsh_k-\bsh_k^0\|^2\nonumber\\
&~~~+\|\bm{x}_k\|^2_{\eta(\beta\bsL+\frac{1}{2}\bsK)
+\eta^2(\frac{\alpha^2}{2}-\alpha\beta+\beta^2)\bsL^2}
+c_2\|\bsh_{k+1}^0-\bsh_{k}^0\|^2\nonumber\\
&~~~-\|\bm{v}_k+\frac{1}{\beta}\bsh_{k}^0\|^2_{\eta(\beta-\frac{1}{2}
-\frac{\eta}{2}-\frac{\eta\beta^2}{2})\bsK}.\label{zero:u3k-1}
\end{align}
Then, from \eqref{zero:u3k-1}, \eqref{zero:gg2}, and \eqref{zero:gg3}, we have
\begin{align}
\bsx_{k+1}^\top\bsK(\bm{v}_{k+1}+\frac{1}{\beta}\bsh_{k+1}^0)
&\le\bm{x}_k^\top(\bsK-\eta\alpha\bsL)(\bm{v}_k+\frac{1}{\beta}\bsh_{k}^0)\nonumber\\
&~~~+\|\bm{x}_k\|^2_{\eta(\beta\bsL+\frac{1}{2}\bsK)+\eta^2(\frac{\alpha^2}{2}
-\alpha\beta+\beta^2)\bsL^2+\frac{3\eta}{2}(1+2\eta)L_f^2\bsK}\nonumber\\
&~~~+\frac{3np}{4}L_f^2\delta^2_{k}\eta(1+2\eta)
+\frac{3np}{4}c_2L_f^2(\delta^2_{k+1}+\delta^2_{k})\nonumber\\
&~~~+3\eta^2c_2L^2_f\|\bar{\bsh}_{k}\|^2-\|\bm{v}_k+\frac{1}{\beta}\bsh_{k}^0\|^2_{\eta(\beta-\frac{1}{2}
-\frac{\eta}{2}-\frac{\eta\beta^2}{2})\bsK}.\label{zero:u3k}
\end{align}

Similar to the way to get \eqref{nonconvex:v4k-zero}, we have
\begin{align}
n(f(\bar{x}_{k+1})-f^*)
&=\tilde{f}(\bar{\bsx}_k)-\tilde{f}^*+\tilde{f}(\bar{\bsx}_{k+1})-\tilde{f}(\bar{\bsx}_k)\nonumber\\
&\le\tilde{f}(\bar{\bsx}_k)-\tilde{f}^*
-\eta\bar{\bsh}_{k}^\top\bsg^0_k
+\frac{\eta^2L_f}{2}\|\bar{\bsh}_{k}\|^2\nonumber\\
&=\tilde{f}(\bar{\bsx}_k)-\tilde{f}^*
-\eta\bar{\bsh}_{k}^\top\bsh^0_k
+\frac{\eta^2L_f}{2}\|\bar{\bsh}_{k}\|^2
-\eta\bar{\bsh}_{k}^\top(\bsg^0_k-\bsh^0_k)\nonumber\\
&\le n(f(\bar{x}_k)-f^*)-\frac{\eta}{4}(1-2\eta L_f)\|\bar{\bsh}_{k}\|^2
+\frac{\eta}{2}\|\bar{\bsh}^0_k-\bar{\bsh}_k\|^2\nonumber\\
&~~~-\frac{\eta}{4}\|\bar{\bsh}_{k}^0\|^2
-\eta\bar{\bsh}_{k}^\top(\bsg^0_k-\bsh^0_k).\label{zero:u4k-1}
\end{align}
Then, from \eqref{zero:u4k-1}, the  Cauchy--Schwarz inequality, \eqref{zero:gradient-close2}, \eqref{zero:gg4}, and \eqref{zero:gg5}, we have
\begin{align}
n(f(\bar{x}_{k+1})-f^*)
&\le n(f(\bar{x}_k)-f^*)-\frac{\eta}{4}(1-2\eta L_f)\|\bar{\bsh}_{k}\|^2
+\frac{\eta}{2}\|\bar{\bsh}^0_k-\bar{\bsh}_k\|^2\nonumber\\
&~~~-\frac{\eta}{4}\|\bar{\bsh}_{k}^0\|^2
+\frac{\eta}{8}\|\bar{\bsh}_{k}\|^2+2\eta\|\bsg^0_k-\bsh^0_k\|^2\nonumber\\
&\le n(f(\bar{x}_k)-f^*)-\frac{\eta}{8}(1-4\eta L_f)\|\bar{\bsh}_{k}\|^2
+\frac{3np}{4}L_f^2\delta^2_{k}\eta\nonumber\\
&~~~ +\|\bsx_k\|^2_{\frac{3\eta}{2} L_f^2\bsK}+\frac{np}{16}L_f^2\delta^2_{k}\eta
-\frac{\eta}{8}\|\bar{\bsg}^0_k\|^2+\frac{np}{2}L_f^2\delta^2_{k}\eta\nonumber\\
&= n(f(\bar{x}_k)-f^*)-\frac{\eta}{8}(1-4\eta L_f)\|\bar{\bsh}_{k}\|^2
+\frac{21np}{16}L_f^2\delta^2_{k}\eta\nonumber\\
&~~~ +\|\bsx_k\|^2_{\frac{3\eta}{2} L_f^2\bsK}-\frac{\eta}{8}\|\bar{\bsg}^0_k\|^2.\label{zero:u4k}
\end{align}

Hence, from \eqref{zero:u1k}--\eqref{zero:u4k}, we have
\begin{align}\label{zero:ukLya}
U_{k+1}
&\le U_{k}-\|\bsx_k\|^2_{\eta\tilde{\bsM}_1-\eta^2\tilde{\bsM}_2}
-\|\bm{v}_k+\frac{1}{\beta}\bsh_k^0\|^2_{\eta\bsM_3-\eta^2\bsM_4}\nonumber\\
&~~~-(\eta\tilde{\epsilon}_{5}-\eta^2\tilde{\epsilon}_{6})\|\bar{\bsh}_{k}\|^2
-\frac{\eta}{8}\|\bar{\bsg}^0_k\|^2
+\epsilon_{11}\delta^2_{k}+\epsilon_{12}\delta^2_{k+1},
\end{align}
where
\begin{align*}
\tilde{\bsM}_1&=(\alpha-\beta)\bsL-\frac{1}{2}(2+9L_f^2)\bsK,\\
\tilde{\bsM}_2&=\beta^2\bsL+(2\alpha^2+\beta^2)\bsL^2+\frac{15}{2}L_f^2\bsK.
\end{align*}

From $\alpha>\beta$, \eqref{nonconvex:KL-L-eq2}, and \eqref{nonconvex:lemma-eq2},  we have
\begin{align}
\tilde{\bsM}=\tilde{\epsilon}_1\bsK~\text{and}~
\tilde{\bsM}_2\le\tilde{\epsilon}_2\bsK.\label{zero:m2}
\end{align}

From \eqref{zero:ukLya}, \eqref{zero:m2}, and \eqref{nonconvex:m3}, we have \eqref{zero:ukLya2}.
\end{proof}

We are now ready to prove Theorem~\ref{zero:thm-determin-sm}.

Similar to the way to get \eqref{nonconvex:vkLya3.2}--\eqref{nonconvex:vkLya3.1}, we know
\begin{align}
U_{k}
&\ge\epsilon_{9}(\|\bsx_{k}\|^2_{\bsK}+\|\bsv_k+\frac{1}{\beta}\bsh_k^0\|^2_{\bsK})
+n(f(\bar{x}_k)-f^*)\label{zero:ukLya3.2}\\
&\ge\epsilon_{9}\hat{U}_k\ge0,\label{zero:ukLya3}
\end{align}
and
\begin{align}\label{zero:ukLya3.1}
U_k\le\epsilon_{8}\hat{U}_k.
\end{align}

From the settings on $\alpha$ and $\beta$, we know that \eqref{nonconvex:beta2} still holds. Similar to the way to get \eqref{nonconvex:beta1} and \eqref{nonconvex:beta3}, we have
\begin{align}
\tilde{\epsilon}_{1}>0~\text{and}~
\tilde{\epsilon}_{5}>0.\label{zero:beta3}
\end{align}
From \eqref{nonconvex:beta2}, \eqref{zero:beta3}, and $0<\eta<\min\{\frac{\tilde{\epsilon}_1}{\tilde{\epsilon}_2},
~\frac{\epsilon_3}{\epsilon_4},~\frac{\tilde{\epsilon}_5}{\tilde{\epsilon}_6}\}$, we know that \eqref{nonconvex:vkLya1.2} still holds. Moreover,
\begin{align}
\eta(\tilde{\epsilon}_1-\eta\tilde{\epsilon}_{2})>0,
~\eta(\tilde{\epsilon}_{5}-\eta\tilde{\epsilon}_{6})>0,~\text{and}~
\tilde{\epsilon}_{7}>0.\label{zero:beta1.2}
\end{align}

From \eqref{zero:ukLya2}, \eqref{nonconvex:vkLya1.2}, \eqref{zero:beta1.2}, and $\bsK\ge0$, we have
\begin{align}\label{zero:ukLya2.1}
U_{k+1}\le U_{k}-\tilde{\epsilon}_{7}(\|\bsx_k\|^2_{\bsK}+\|\bar{\bsg}^0_k\|^2)
+\epsilon_{11}\delta^2_{k}+\epsilon_{12}\delta^2_{k+1}.
\end{align}
Hence, summing \eqref{zero:ukLya2.1} over $k=0,\dots,T$ yields
\begin{align}\label{zero:ukLya2.2}
U_{T+1}+\tilde{\epsilon}_{7}\sum_{k=0}^{T}(\|\bsx_k\|^2_{\bsK}+\|\bar{\bsg}^0_k\|^2)\le U_{0}
+(\epsilon_{11}+\epsilon_{12})\sum_{k=0}^{T+1}\delta^2_{k}.
\end{align}

We know
\begin{align}\label{zero:determin-delta2}
\delta^2_{k}=(\max_{i\in[n]}\{\delta_{i,k}\})^2\le \sum_{i=1}^{n}\delta_{i,k}^2.
\end{align}
From \eqref{zero:determin-delta} and \eqref{zero:determin-delta2}, we have
\begin{align}\label{zero:determin-delta3}
\sum_{k=0}^{T+1}\delta^2_{k}\le \sum_{i=1}^{n}\delta^a_{i},~\forall T\in\mathbb{N}_0.
\end{align}

From \eqref{zero:ukLya2.2}, \eqref{zero:determin-delta3}, \eqref{zero:ukLya3}, \eqref{zero:ukLya3.1}, and \eqref{zero:beta1.2}, we have \eqref{zero:thm-determin-sm-equ1}.

From \eqref{zero:ukLya2.2}, \eqref{zero:determin-delta3}, \eqref{zero:ukLya3.1}, and \eqref{zero:beta1.2}, we have \eqref{zero:thm-determin-sm-equ2}.

\subsection{Proof of Theorem~\ref{zero:thm-determin-ft}}\label{zero:proof-thm-determin-ft}
In this proof, we use the same notations as those used in the proof of Theorem~\ref{zero:thm-determin-sm}.

(i) From \eqref{zero:beta1.2} and \eqref{nonconvex:vkLya1.2}, we have
\begin{align}\label{zero:beta4}
\tilde{\epsilon}_{10}>0~\text{and}~\tilde{\epsilon}
=\frac{\tilde{\epsilon}_{10}}{\epsilon_8}>0.
\end{align}

Similar to the way to get \eqref{nonconvex:vkLya-epsilon}, we have
\begin{align}\label{zero:ukLya-epsilon}
0<\tilde{\epsilon}<\frac{1}{8}.
\end{align}

From \eqref{zero:ukLya2}, \eqref{zero:beta1.2}, and \eqref{nonconvex:gg3}, we have
\begin{align}\label{zero:ukLya4}
U_{k+1}\le U_{k}-\tilde{\epsilon}_{10}\hat{U}_k+
\epsilon_{11}\delta^2_{k}+\epsilon_{12}\delta^2_{k+1}.
\end{align}

From \eqref{zero:ukLya4}, \eqref{zero:ukLya3.1}, and \eqref{zero:beta4}, we have
\begin{align}\label{zero:ukLya5}
U_{k+1}
\le (1-\tilde{\epsilon})U_k+\epsilon_{11}\delta^2_{k}+\epsilon_{12}\delta^2_{k+1}.
\end{align}

From \eqref{zero:ukLya5}, \eqref{zero:ukLya-epsilon}, and \eqref{zero:ukLya3}, we have
\begin{align}\label{zero:ukLya6}
U_{k+1}
&\le (1-\tilde{\epsilon})^{k+1}U_0+\epsilon_{11}\sum_{\tau=0}^k(1-\tilde{\epsilon})^\tau\delta^2_{k-\tau}
+\epsilon_{12}\sum_{\tau=0}^k(1-\tilde{\epsilon})^\tau\delta^2_{k+1-\tau}.
\end{align}

From $\delta_{i,k}\in(0,\hat{\epsilon}^{\frac{k}{2}}]$, \eqref{zero:ukLya6}, and \eqref{zero:ukLya3.1}, we have
\begin{align}\label{zero:ukLya6.1}
U_{k+1}
&\le(1-\tilde{\epsilon})^{k+1}U_0+(\frac{\epsilon_{11}}{1-\tilde{\epsilon}}+\epsilon_{12})
\sum_{\tau=0}^{k+1}(1-\tilde{\epsilon})^\tau\delta^2_{k+1-\tau}\nonumber\\
&\le(1-\tilde{\epsilon})^{k+1}\epsilon_8\hat{U}_0+(\frac{\epsilon_{11}}{1-\tilde{\epsilon}}+\epsilon_{12})
\sum_{\tau=0}^{k+1}(1-\tilde{\epsilon})^\tau\hat{\epsilon}^{k+1-\tau}.
\end{align}

From  $\hat{\epsilon}\in(0,1)$, \eqref{zero:beta4}, \eqref{zero:ukLya-epsilon}, and \eqref{zerosg:lemma:sumgeo-equ}, we have
\begin{align}\label{zero:ukLya7}
U_{k+1}
\le(1-\tilde{\epsilon})^{k+1}\epsilon_8\hat{U}_0
+\phi(\tilde{\epsilon},\hat{\epsilon},\breve{\epsilon}).
\end{align}

From \eqref{zero:ukLya3.2}, we have
\begin{align}\label{zero:ukLya8}
\|\bsx_{k}-\bar{\bsx}_k\|^2+n(f(\bar{x}_k)-f^*)
\le\hat{U}_k\le\frac{U_{k}}{\epsilon_{9}}.
\end{align}

Hence, \eqref{zero:ukLya7} and \eqref{zero:ukLya8}
give \eqref{zero:thm-determin-ft-equ1}.

(ii) From \eqref{zero:thm-determin-ft-equ1} and \eqref{nonconvex:vkLya6} we know \eqref{zero:thm-determin-ft-equ2} holds.

\end{document}